\documentclass[12pt,a4paper]{amsart}
\usepackage[english]{babel}
\usepackage[left=2cm,right=2cm,top=2cm,bottom=2cm]{geometry}
\usepackage[numbers,sort&compress]{natbib} 
\usepackage{enumitem}
\usepackage{hyperref}
\usepackage{mathtools}
\usepackage{mathrsfs}
\usepackage{esint}
\usepackage{cases}
\usepackage{xcolor}
\usepackage{comment}

\theoremstyle{definition}
\newtheorem{defi}{Definition}[section]
\newtheorem{rem}[defi]{Remark}

\theoremstyle{plain}
\newtheorem{teo}[defi]{Theorem}
\newtheorem{lem}[defi]{Lemma}

\newcommand{\N}{\mathbb{N}}
\newcommand{\Z}{\mathbb{Z}}

\newcommand{\W}{\mathbb{W}}
\newcommand{\vp}{\varphi}
\newcommand{\R}{\mathbb{R}}
\newcommand{\pp}{\partial}
\newcommand{\se}{\subseteq}
\newcommand{\ceq}{\coloneqq}

\newcommand{\res} {\mathop{\hbox{\vrule height 7pt width .5pt depth 0pt \vrule height .5pt width 6pt depth 0pt}}\nolimits}
\newcommand{\leb}{\mathcal{L}}
\newcommand{\sbv}{\operatorname{SBV}}
\newcommand{\bv}{\operatorname{BV}}
\newcommand{\J}{\mathcal{J}}
\newcommand{\hau}{\mathcal{H}}
\newcommand{\nl}{\left\|}
\newcommand{\nr}{\right\|}
\newcommand{\loc}{\operatorname{loc}}
\newcommand{\ap}{{\operatorname{ap}}}
\newcommand{\s}{\mathcal{S}}

\renewcommand{\H}{\mathbb{H}}
\newcommand{\ve}{\varepsilon}
\newcommand{\Id}{\operatorname{Id}}

\newcommand{\C}{\mathcal {C}}

\newcommand{\Leb}{\leb^{2n+1}}
\newcommand{\Hn}{\H^n}
\newcommand{\diam}{\operatorname{diam}}
\newcommand{\dist}{\operatorname{dist}}
\newcommand{\Q}{\mathcal{Q}}

\subjclass[2020]{26B30, 53C17, 49Q15, 28A75.}

\title[A BMO formula for SBV functions in Heisenberg groups]{A BMO-type representation formula for the total variation of SBV functions in Heisenberg groups}

\keywords{Special functions with bounded variation, bounded mean oscillation, Carnot-Carathéodory spaces, sub-Riemannian geometry, rectifiable sets.}

\thanks{M.D.M. acknowledges the support of the Swiss National Science Foundation (SNSF) Starting Grant \emph{Challenges and Breakthroughs in the Mathematics of Plasmas}, TMSGI2\textunderscore226018. D.V. was  supported by the University of Padova. M.D.M. and D.V. are members of GNAMPA of INdAM; in particular, D.V. receives funding from the INdAM project VAC\&GMT as well as from the INdAM-GNAMPA 2026 Project Variational, Geometric, and Analytic Perspectives on Regularity, CUP E53C25002010001.}

\begin{document}

\author[M. Di Marco]{Marco Di Marco}
\address[M. Di Marco]{ETH Z\"urich, Department of Mathematics, R\"amistrasse 101, 8092 Z\"urich, Swit\-zer\-land.}

\email{mdimarco@ethz.ch}

\author[S. Don]{Sebastiano Don}
\address[S. Don]{Dipartimento di Ingegneria Civile, Architettura, Territorio, Ambiente e di Matematica (DICATAM), via Branze, 38, 25123 Brescia, Italy}
\email{sebastiano.don@unibs.it}

\author[D. Vittone]{Davide Vittone}
\address[D. Vittone]{Dipartimento di Matematica ``T. Levi-Civita'', Università di Padova, via Trieste 63, 35121 Padova, Italy.}
\email{davide.vittone@unipd.it}

\begin{abstract}
We prove a formula for the total variation of SBV functions in sub-Riemannian Heisenberg groups arising as the limit of certain BMO-type functionals.
\end{abstract}

\maketitle

\section{Introduction}
A few years ago L.~Ambrosio, J.~Bourgain, H.~Brezis and A.~Figalli~\cite{abbf16} introduced the functionals 
\[
\kappa_r(u) \ceq r^{n-1} \sup_{\mathcal G_r}\ \sum_{\Q\in\mathcal G_r}\fint_\Q \left|u(x)-\fint_\Q u\right|dx,\qquad u\in L^1(\R^n),
\]
where  $r>0$ is fixed and the supremum is taken among all families $\mathcal G_r$ of disjoint open cubes $\Q$ of side length $r$ and arbitrary orientation. The functionals $\kappa_r$ are   inspired by the space $\operatorname{BMO}$ of functions with bounded mean oscillation, see also~\cite{BBM,Davila_CalcVar2002,PonceSpector_NLA2017,fms18,add1,add2,add3,add4,add5,add6,add7,add8}. It was conjectured in~\cite{abbf16} that, if $u\in \sbv_{\loc}(\R^n)$ belongs to the (local) class of {\em special functions with bounded variation}, then\footnote{We denote by $|\mu|$ the total variation of a vector-valued Radon measure $\mu$.}
\begin{equation}\label{eq_limkaepSBV}
\lim_{r\to0^+} \kappa_r(u)=\frac14\int_{\R^n}|D^{\ap} u|\,d\leb^n + \frac 12 |D^j u|(\R^n).
\end{equation}
Recall (see e.g.~\cite{afp}) that for $u\in \sbv_{\loc}(\R^n)$ the derivative $Du$ in the sense of distributions is a vector-valued measure that can be decomposed as $Du=D^{\ap} u\leb^n+D^ju$, where $D^{\ap} u\in L^1_{\loc}(\R^n)$ is the {\em approximate gradient} of $u$ and  the {\em singular part} $D^ju$  (with respect to $\leb^n$)  of $Du$ is concentrated on the jump set $\J_u$, which is a countably $(n-1)$-rectifiable set. 
The equality~\eqref{eq_limkaepSBV} was proved in~\cite{dpfp} by exploiting the outcomes of~\cite{fms16} and some new approximation (with respect to the $\bv$ norm) results for $\sbv$ functions. A $\Gamma$-convergence approach was developed in~\cite{ARBDN_Indiana,ln24}.  Formula~\eqref{eq_limkaepSBV} reveals one of the most interesting features of the functionals $\kappa_r$, as it provides a formula for the total variation of $u$ that is independent of the theory of distributions: remarkably, no derivatives of $u$ are involved -- just integrals.

In this paper we develop a similar program in the genuinely non-Euclidean setting\footnote{See also the related works \cite{lpz24,lpz24bis,scz26}.} of the sub-Riemannian Heisenberg groups $\H^n$. Recall that the $n$-th Heisenberg group $\H^n\equiv \R^{2n+1}$ is a connected and simply connected Lie group whose Lie algebra $\mathfrak h^n$ is 2-step nilpotent and admits the stratification $\mathfrak h^n=\mathfrak h^n_1 \oplus \mathfrak h^n_2$. The directions in $\mathfrak h^n_1$ are called {\em horizontal}; their Lie brackets generate the Lie algebra and determine the sub-Riemannian structure of $\H^n$, see Section~\ref{subsecHeis}.

The space $\bv_\H$ of functions with horizontal bounded variation in $\H^n$ is introduced (\cite{cdg,fssc}) as the space of $L^1$ functions $u:\H^n\to\R$ whose horizontal derivatives, in the sense of distributions, are represented by an $\R^{2n}$-vector valued Radon measure $D_\H u$. It is known (see~\cite{dv,sy03}) that $D_\H u$ admits the decomposition 
\[
D_\H u=D_\H^{\ap}u\Leb+D_\H^cu+D_\H^ju,
\]
where $D_\H^{\ap}u$ is the {\em approximate $\H$-gradient} of $u$, $D_\H^cu$ is called {\em Cantor part} of $D_\H u$, and $D_\H^j u$ is the {\em jump part} of $D_\H u$; recall that $\Leb$ is a Haar measure on $\H^n$ when the latter is identified with $\R^{2n+1}$ by standard exponential coordinates. 

In Section~\ref{subsec_cubi} we introduce, for every side parameter $r>0$, some natural families  of {\em Heisenberg $r$-cubes}; here, the adjective {\em natural} refers to the fact that they behave well with respect to left-translations, homogeneous dilations, and  isometries such as Heisenberg rotations. In the spirit of~\cite{abbf16}, given an open subset $\Omega\se\H^n$, a function $u\in L^1(\Omega)$ and $r>0$, it is then natural to consider the functional
  \[
    k_r(u) \ceq r^{2n+1} \sup_{\mathcal G_r}\   \sum_{\Q\in \mathcal G_r} \fint_{\Q} \left| u(x) -\fint_{\Q}  u\right|dx,
    \]
where $\mathcal G_r$ varies among all the possible families of pairwise disjoint Heisenberg $r$-cubes  contained in $\Omega$. We can now state our main result.

\begin{teo}\label{teo_intromain}
Let $\Omega \se \H^n$ be an open set and $u \in \sbv_\H(\Omega)$. Then
\[
\lim_{r \to 0} k_r(u) = \frac{1}{4} \int_\Omega |D^\ap_\H u|d\leb^{2n+1}+\frac{1}{2}|D_\H^ju|(\J_u).
\]
\end{teo}

The proof of Theorem~\ref{teo_intromain}, which occupies Section~\ref{sec_dimostrazione}, follows the strategy outlined in~\cite{fms16,dpfp}
and  is based on the results of the previous paper~\cite{sbvx} of the authors. More precisely, the proof  relies crucially on~\cite[Theorem~1.4]{sbvx} which states that any $\sbv_\H$ function can be approximated {\em with respect to the $\bv_\H$ norm} by a sequence of {\em good} $\sbv_\H $ functions, i.e. (see Definition~\ref{def_goodsbv}) by $\sbv_\H$ functions $u$ such that (i) the jump set $\J_u$ is compact and contained in an intrinsic $C^1$ hypersurface, and (ii) $u$ is smooth on the complement of $\J_u$. Using this fact, one can easily prove (Theorem~\ref{teo_bastanolegood}) that the problem reduces to proving Theorem~\ref{teo_intromain} for good $\sbv_\H$ functions only:  eventually, for doing so one needs to make use of certain fine properties of good $\sbv_\H$ functions that are stated and proved in the pivotal Lemma~\ref{lem_comb}.
\medskip

\textbf{Structure of the paper} In Section~\ref{sec_preliminari} we introduce all the needed preliminary facts about Heisenberg groups, $\bv_\H$ functions and Heisenberg cubes. Section~\ref{sec_goodSBV} is devoted to good $\sbv_\H$ functions and, in particular, to the proof of Lemma~\ref{lem_comb}. Section~\ref{sec_dimostrazione} contains the proof of Theorem~\ref{teo_intromain}.
\medskip

\textbf{Statement of AI use.} The authors used AI tools for minor copyediting assistance during the final revision of this paper. The manuscript and all mathematical content were conceived and written entirely by the authors.

\section{Notation and preliminary results}\label{sec_preliminari}

In this section we introduce Heisenberg groups, $\bv_\H$ functions and the relevant notation. We also introduce Heisenberg cubes and some properties.

\subsection{Heisenberg groups}\label{subsecHeis}
\begin{defi}
For $n\geq 1$ we denote by $\H^n$ the $n$-th {\em Heisenberg group}, identified with $\R^{2n+1}$ through exponential coordinates. We denote a point $\H^n \ni p=(z,t)$ by $z \in \R^{2n}$ and $t \in \R$. If $p=(z,t),q=(z',t') \in \H^n$, the group operation is defined as 
\[
p \cdot q \ceq (z+z',t+t'+2\omega(z,z')),
\qquad
\omega(z,z') \ceq \sum_{j=1}^n(x'_jy_j-x_jy'_j),
\]
where $z=(x,y), z'=(x',y')$ with $x,y,x',y' \in \R^n$. If $\H^n \ni p=(z,t)$, its inverse is $p^{-1}=(-z,-t)$ and $0=(0,0)\in \H^n$ is the identity of $\H^n$. For $\lambda >0$, we denote by $\delta_\lambda\colon\H^n\to\H^n$ the dilations of the Heisenberg groups defined, for $p=(z,t)\in\H^n$, by $\delta_\lambda(z,t) \ceq (\lambda z, \lambda^2 t)$. Observe that dilations form a one-parameter family of group isomorphisms.

The {\em homogeneous dimension} of $\H^n$ is $2n+2$. The Lebesgue measure $\mathcal L^{2n+1}$ is the Haar measure on $\H^n\equiv\R^{2n+1}$ and it is $(2n+2)$-homogeneous with respect to dilations. The $(2n+1)$-dimensional Lie algebra $\mathfrak{h}^n$ of left invariant vector fields is generated by the vector fields $X_1,\dots,X_n,Y_1,\dots,Y_n,T$ where (for $1 \leq j \leq n$)
\[
X_j \ceq \pp_{x_j}+2y_j\pp_t,
\qquad
Y_j \ceq \pp_{y_j}-2x_j\pp_t,
\qquad
T \ceq \pp_t.
\]
We denote by $\mathfrak h^n_1$ the horizontal subspace of $\mathfrak h^n$, i.e.,
\[
\mathfrak h^n_1 \ceq \operatorname{span}(X_1,\dots,X_n,Y_1,\dots,Y_n),
\]
and by $\mathfrak h^n_2$ the linear span of $T$; the Lie algebra $\mathfrak h^n$ admits the 2-step stratification $\mathfrak h^n=\mathfrak h^n_1 \oplus \mathfrak h^n_2$. If $p \in \H^n$ we denote by $H_p\H^n \ceq \operatorname{span}(X_1(p),\dots,X_n(p),Y_1(p),\dots,Y_n(p))$ and by $H \H^n \ceq \bigcup_{p \in \H^n}H_p\H^n$ the horizontal fiber bundle. We denote by $\langle \cdot, \cdot \rangle$ the inner product on $\mathfrak h^n$ that makes the basis $X_1,\dots,X_n,Y_1,\dots,Y_n,T$ orthonormal and by $|\cdot|$ the corresponding norm. With a slight abuse of notation, we also denote by $\langle \cdot, \cdot \rangle$ (resp. by $|\cdot|$) the restriction of the inner product (resp. of the norm) on the horizontal subspace $\mathfrak h^n_1$.
\end{defi}

\begin{defi}
 We define the \emph{Carnot-Carathéodory} distance $d$ for $p,q \in \H^n$ as
\[
d(p,q)\ceq \inf\left\{\|h\|_{L^1([0,1],\R^{2n})} : 
\begin{array}{l}
\text{the curve $\gamma_h:[0,1]\to\H^n$ defined by}\\
\gamma_h(0)=p,\ \dot\gamma_h=\sum_{j=1}^n(h_jX_j+h_{j+n}Y_j)(\gamma_h)\\
\text{has final point }\gamma_h(1)=q
\end{array}
\right\}.
\]
Then we denote the corresponding open balls of center $p$ and radius $r$ as $B(p,r)$ and, for $m \in \N$, by $\hau^m$ the corresponding $m$-dimensional Hausdorff measure and  by $\s^m$ the corresponding $m$-dimensional spherical Hausdorff measure.
\end{defi}

\subsection{Functions with bounded variation in Heisenberg groups}

\begin{defi}
Let $\Omega \se \H^n$ be an open set. We say that $u \in L^1_{\loc}(\Omega)$ is a \emph{function of locally bounded $\H$-variation}, and we write $u \in \bv_{\H,\loc}(\Omega)$, if there exists a $\R^{2n}$-valued Radon measure $D_\H u=(D_{X_1}u,\dots, D_{Y_n}u)$ on $\Omega$ such that, for every open set $A \subset \subset \Omega$, for every $1 \leq i \leq n$ and for every $\vp \in C^1_c(A)$ one has 
\[
\int_A \vp d(D_{X_i}u)=-\int_A u X_i\vp d \leb^{2n+1} \quad \text{ and }\quad \int_A \vp d(D_{Y_i}u)=-\int_A u Y_i\vp d \leb^{2n+1}.
 \]
  Moreover, if $u \in L^1(\Omega)$ and $D_\H u$ has bounded total variation $|D_\H u|$, then we say that $u$ has \emph{bounded $\H$-variation} and we write $u \in \bv_\H(\Omega)$. 

For every $u \in \bv_\H(\Omega)$ we define the norm
\[
\nl u \nr_{\bv_\H(\Omega)} \ceq \nl u \nr_{L^1(\Omega)}+|D_\H u|(\Omega).
\]
The space $\bv_\H(\Omega)$ equipped with the above norm is a Banach space.
\end{defi}

\begin{defi}
Let $\Omega \se \H^n$ be an open set. For every $u \in \bv_\H(\Omega)$ we decompose
\[
D_\H u=D_\H ^au+D_\H ^su
\]
where $D_\H^au$ and $D^s_\H u$ denote, respectively, the \emph{absolutely continuous part} and the \emph{singular part} of $D_\H u$ with respect to the Lebesgue measure $\leb^{2n+1}$.
\end{defi}

It is worth recalling that $\leb^{2n+1}$, $\hau^{2n+2}$ and $\s^{2n+2}$ are Haar measures on $\H^n$; in particular, they differ by multiplicative constants.

\begin{defi}
    We say that a measurable set  $E \se \H^n$ has \emph{locally finite $\H$-perimeter} (resp. \emph{finite $\H$-perimeter}) in an open set $\Omega \se \H^n$ if its characteristic function $\chi_E$ belongs to $\bv_{\H,\loc}(\Omega)$ (resp. $\chi_E \in \bv_\H(\Omega)$). 
    In such a case we define the \emph{$\H$-perimeter measure} $P_\H(E, \cdot)$ of $E$ as $P_\H(E, \,\cdot\,):=|D_\H \chi_E|$.
\end{defi}

\begin{defi}\label{def_c1x}
Let $\Omega \se \H^n$ be an open set and $f\colon \Omega \to \R$. We say that $f \in C^1_\H(\Omega)$ if $f$ is continuous and its \emph{horizontal gradient}  $\nabla_\H f \ceq (X_1f,\dots,Y_nf)$,   in the sense of distributions, is represented by a continuous function.
\end{defi}

\begin{defi}
 Let $\Omega \se \H^n$ be an open set,   $u \in L^1_{\loc}(\Omega)$, $z \in \R$ and $p \in \Omega$. We say that $z$ is the \emph{approximate limit} of $u$ at $p$ if
    \[
    \lim_{r \to 0} \fint_{B(p,r)}|u-z|d\leb^{2n+1}=0.
    \]
If the approximate limit of $u$ at $p$ exists, it is also unique (see \cite[Definition 2.19]{dv}). We hence denote by $u^\star(p)$ the approximate limit of $u$ at $p$ and by $\mathcal{S}_u$ the subset of points in $\Omega$ where $u$ does not admit an approximate limit. 
\end{defi}

\begin{defi}\label{def_approxdiff}
    Let $\Omega \se \H^n$ be an open set, $u \in L^1_{\loc}(\Omega)$ and $p \in \Omega \setminus \mathcal{S}_u$. We say that $u$ is \emph{approximately $\H$-differentiable} at $p$ if there exist a neighbourhood $U\subset \Omega$ of $p$ and $f \in C^1_\H(U)$ such that $f(p)=0$ and
    \[
    \lim_{r \to 0}\fint_{B(p,r)} \frac{|u-u^\star (p)-f|}{r}d \leb^{2n+1} =0.
    \]
    The set of points in $\Omega$ where $u$ is approximately $\H$-differentiable is denoted by $\mathcal{D}_u$.
The vector $\nabla_\H f(p) \in \R^{2n}$ is uniquely determined (see~\cite[Proposition 2.30]{dv}): we call it \emph{approximate $\H$-gradient} of $u$ at $p$ and we denote it by $D^\ap_\H u(p)$.
\end{defi}

We  now  spend a few words about intrinsically $C^1$ (or $C^1_\H$) hypersurfaces and the notion of $\H$-rectifiability.

\begin{defi}\label{def_ipersup}
We say that $S \se \H^n$ is a \emph{$C^1_\H$-hypersurface} if for every $p \in S$ there exist $r>0$ and $f \in C^1_\H(B(p,r))$ such that the following facts hold:
\begin{enumerate}
\item[(i)]$ S \cap B(p,r)=\lbrace q \in B(p,r):f(q)=0 \rbrace$,
\item[(ii)]$\nabla_\H f\neq 0$ on $B(p,r)$.
\end{enumerate}
We define the \emph{horizontal normal} to $S$ at $p \in S$ as
$\nu_S(p) \ceq \frac{\nabla_\H f(p)}{|\nabla_\H f(p)|}$. Notice that $\nu_S(p)$ is well defined up to a sign and that it does not depend on the choice of $f$, see \cite[Corollary 2.14]{dv}.
\end{defi}

\begin{defi}\label{def_Xrectifiable}
Let $S \se \H^n$. We say that $S$ is \emph{countably $\H$-rectifiable}\footnote{Since the involved dimension will always be $2n+1$, for the sake of shortness we avoid the more correct, but heavier, terminology {\em countably $(2n+1)$-$\H$-rectifiable}.} if there exists a family $\lbrace S_h: h \in \N \rbrace$ of $C^1_\H$-hypersurfaces such that
\[
\hau^{2n+1}\left( S \setminus \bigcup_{h \in \N}S_h \right)=0.
\]
  Moreover, if $\hau^{2n+1}(S)<+\infty$, we say that $S$ is \emph{$\H$-rectifiable}. We define the \emph{horizontal normal} to a countably $\H$-rectifiable set $S$ at $p \in S$ as 
\[
\nu_S(p) \ceq \nu_{S_h}(p) \text{ if }p \in S_h \setminus \bigcup_{k<h}S_k.
\]
Notice that $\nu_S$ is well defined, up to a sign, $\hau^{2n+1}$-a.e., see \cite[Proposition 2.18]{dv}.
\end{defi}

\begin{defi}\label{def_b+b-}
Fix $p \in \H^n$ and $\nu \in \mathbb{S}^{2n-1}$. Let $f_{p,\nu} \in C^1_\H(\H^n)$ be the unique function such that $f_{p,\nu}(p)=0$ and $\nabla_\H f_{p,\nu} \equiv \nu$. For every $r>0$ we define
\[
B^\pm_\nu(p,r) \ceq B(p,r) \cap \lbrace \pm f_{p,\nu}>0\rbrace.
\]
\end{defi}

By using the \emph{half-balls} from Definition \ref{def_b+b-}, one can define intrinsic approximate jumps.

\begin{defi}\label{def_approxjump}
 Let $\Omega \se \H^n$ be an open set, $u \in L^1_{\loc}(\Omega)$ and $p \in \Omega$. We say that $u$ has an \emph{approximate $\H$-jump} at $p$ if there exist $u^+,u^- \in \R$ with $u^+ \neq u^-$ and $\nu \in \mathbb{S}^{2n-1}$ such that
\begin{equation}\label{eq_xjump}
\lim_{r \to 0} \fint_{B^+_\nu(p,r)} |u-u^+|d\leb^{2n+1}=\lim_{r \to 0} \fint_{B^-_\nu(p,r)} |u-u^-|d\leb^{2n+1}=0.
\end{equation}
The \emph{jump set} $\J_u$ is defined as the set of points where $u$ has an approximate $\H$-jump.

It is worth remarking that, by~\cite[Theorem 1.2 and Remark 2.25]{dv}, the Hausdorff measure $\hau^{2n+1}$ on the jump set is $\sigma$-finite.
\end{defi}

\begin{defi}
 Let $\Omega \se \H^n$ be an open set and $u \in \bv_\H(\Omega)$. We define the \emph{jump part} of $D_\H u$ as 
\[
D_\H^ju \ceq D_\H^s u \res \J_u
\]
and the \emph{Cantor part} of $D_\H u$ as
\[
D_\H ^cu \ceq D_\H^s u \res (\Omega \setminus \J_u).
\]
 We say that $u$ is a \emph{special function of bounded $\H$-variation}\footnote{Let us observe that the definition is different from the one provided in \cite{sbvx} since, in Heisenberg groups, the jump set $\J_u$ is always countably $\H$-rectifiable, see for instance \cite{dv,sy03} as well as the recent~\cite{dimarco2026whitneyextensiontheoremintrinsic}.}, and we write $u \in \sbv_\H(\Omega)$, if $D_\H^cu=0$.
\end{defi}

\subsection{Heisenberg cubes}\label{subsec_cubi}

\begin{defi}
In the following we always identify $\H^n \equiv \R^{2n+1}$ by means of exponential coordinates and we will use the following notation:
\[
 \C_H \ceq \left(-\frac12,\frac12\right)^{2n},
 \qquad
 \C \ceq \C_H\times\left(-\frac12,\frac12\right),
\]
i.e., $\C \subset \H^n$ is the unit Euclidean cube centered at $0 \in \H^n$. Clearly, $\leb^{2n+1}(\C)=1$.   We denote by $U(n)$ the group of unitary rotations of $\mathbb C^n\equiv\R^{2n}$.  For $U\in U(n)$ we define the automorphism
\[
 \Phi_U(z,t) \ceq (Uz,t).
\]
The map $\Phi_U$ is an isometry of $\H^n$ for every $U \in U(n)$, see for instance \cite[Section 7]{kn10}. For $p\in\H^n$, $r>0$ and $U\in U(n)$ we define
\[
 T_{p,r,U}(z,t)\ceq p\cdot\delta_r(\Phi_U(z,t)),
\]
and
\begin{equation}\label{eq:cubes}
 \C_{r,U}(p) \ceq  T_{p,r,U}(\C).
\end{equation}
Then $\leb^{2n+1}(\C_{r,U}(p))=r^{2n+2}$.  By compactness of $\overline \C$ and because $\C$ is a neighborhood of the identity, there are dimensional constants $0<2\mathfrak a<\mathfrak b<\infty$ such that
\begin{equation}\label{eq:cube-ball}
 B(p,\mathfrak ar)\subset \C_{r,U}(p)\subset B(p,\tfrac{\mathfrak b}{2}r),
 \quad
 \operatorname{diam}\bigl(\overline{\C_{r,U}(p)}\bigr)\le \mathfrak br
 \quad\text{for all }p \in \H^n,r>0,U \in U(n).
\end{equation}
We denote by $\mathcal T_r$ the collection of all elements in \eqref{eq:cubes}, namely
\[
\mathcal T_r\coloneqq\{\C_{r,U}(p): p\in \H^n, U\in U(n)\}.
\]
We call each element of $\mathcal T_r$ a \emph{Heisenberg $r$-cube}.
\end{defi}
\begin{rem}\label{rem_mathcalP}
The family of cubes $\{k\cdot \C:k\in\Z^{2n+1}\}$ is a paving of $\Hn$: namely, a collection of pairwise disjoint sets whose union covers the whole $\H^n$ up to a $\leb^{2n+1}$-null set. Indeed, away from the coordinate boundaries, for $p=(z,t)$ there is a unique $k_H\in\Z^{2n}$ such that $z-k_H\in \C_H$ and then a unique $k_t\in\Z$ such that
\[
 t-2\omega(k_H,z-k_H)-k_t\in\left(-\frac12,\frac12\right).
\]
Thus $p=(k_H,k_t)\cdot(z-k_H,t-2\omega(k_H,z-k_H)-k_t)$; existence may fail only on a countable union of Euclidean hyperplanes, which is Lebesgue-null.  Accordingly, for $r>0$ and $U\in U(n)$ we set
\begin{equation}\label{eq:pavage}
 \mathcal P_{r,U} \ceq \left\{\delta_r\Phi_U(k\cdot \C):k\in\Z^{2n+1}\right\}.
\end{equation}
Every element of $\mathcal P_{r,U}$ belongs to $\mathcal T_r$. 
    
\end{rem}

We are finally ready to properly define the functional $k_r(\cdot)$ which appears in Theorem \ref{teo_intromain}. Before doing so, we recall the following useful notation:  for every measurable set $E \se \H^n$ with $\leb^{2n+1}(E)\neq0$ and $u \in L^1(E)$ we use, here and in the following, the notation
    \[
    (u)_E \ceq \fint_E u d\leb^{2n+1}=\frac{1}{\leb^{2n+1}(E)}\int_E ud\leb^{2n+1}.
    \]

\begin{defi}

    Let $\Omega \se \H^n$ be an open set, $u \in L^1(\Omega)$ and $r>0$. We define
    \[
    k_r(u) \ceq r^{2n+1} \sup_{\mathcal G_r} \sum_{\Q \in \mathcal G_r} \fint_{\Q} \left| u- (u)_{\Q}\right|d\leb^{2n+1},
    \]
    where $\mathcal G_r$ varies among all the possible subsets of $\mathcal T_r$ such that the Heisenberg $r$-cubes in $\mathcal G_r$ are pairwise disjoint and contained in $\Omega$.
\end{defi}

The following lemma is a strengthened version of \cite[Lemma 4.1]{dmi26} in Heisenberg groups.

\begin{lem}\label{lem:pavage}
Let $A\se \Hn$ be a bounded and open set, $U \in U(n)$, and let
\[
 A_{r,U} \ceq \bigcup\{Q\in\mathcal P_{r,U}:Q\subset A\}.
\]
Then
\[
 \Leb(A\setminus A_{r,U})\xrightarrow{r\to 0}0
\]
uniformly with respect to $U\in U(n)$.
\end{lem}

\begin{proof}
Let $N_{r,U}$ be the union of the boundaries of the tiles in $\mathcal P_{r,U}$.  Then $\Leb(N_{r,U})=0$.  If $p\in A\setminus(A_{r,U}\cup N_{r,U})$ and $\Q\in\mathcal P_{r,U}$ is the unique tile containing $p$, then $\Q\not\subset A$.  Hence
\[
 \dist(p,\Hn\setminus A)\le\diam(\Q)\le\mathfrak b r,
\]
where $\mathfrak b$ is the same dimensional constant appearing in \eqref{eq:cube-ball}. It follows, up to the null set $N_{r,U}$, that
\[
 A\setminus A_{r,U}
 \subset\{p\in A:\dist(p,\Hn\setminus A)\le\mathfrak b r\}.
\]
The sets on the right hand side of the above inclusion decrease to the empty set as $r\to 0$ and have finite measure.  The conclusion follows from continuity from above, uniformly in $U$.
\end{proof}

In the following lemmata we collect the preliminary results that we are going to need in order to prove our main result, Theorem \ref{teo_intromain}.

\begin{lem}\label{lem:NormalOnCube}
For every $\nu\in\mathbb S^{2n-1}$,
\begin{equation}\label{eq:phi-cube}
\int_{\C}|\langle (\nu,0), p\rangle| dp=\int_{\C_H}|\langle\nu,z\rangle|dz\le\frac14.
\end{equation}
Equality holds if and only if $\nu=\pm e_j$ for some $j\in\{1,\ldots,2n\}$.
\end{lem}

\begin{proof}
This is~\cite[Lemma~3.1]{fms16}, applied in dimension $2n$ to $\C_H$.
\end{proof}
\begin{lem}\label{lem_14xi}
    For every horizontal vector $\xi \in \R^{2n}$, $U \in U(n), p \in \Hn$ and $r>0$ we have
    \begin{itemize}
    \item[(i)] $\displaystyle\int_\C |\langle \xi,Uz \rangle| dzdt\leq \frac{1}{4}|\xi|;$
    \item[(ii)]$\displaystyle\int_{\C_{r,U}(p)}\left| \left\langle \xi,\pi_H(p^{-1}\cdot q) \right\rangle \right| d\leb^{2n+1}(q) \leq \frac{1}{4}|\xi|r^{2n+3},$
    \end{itemize}
    where $\pi_H:\H^n \to \R^{2n}$ denotes the usual horizontal projection. Equality holds in $(i)$ and $(ii)$ if and only if $\xi= \pm |\xi|Ue_j$ for some $j=1,\dots,2n$.
\end{lem}
    \begin{proof}
    We may assume without loss of generality that $\xi\neq0$. 
       To prove $(i)$ it is enough to apply Lemma~\ref{lem:NormalOnCube} with the choice $\nu=\frac{U^*\xi}{|\xi|}$. It is clear that if $\xi=\pm |\xi| Ue_j$, the equality holds in $(i)$. On the other hand, if equality holds in $(i)$ and $\xi\neq 0$ we find 
       \[
        \int_\C\left|\left\langle \left(U^*\tfrac{\xi}{|\xi|},0\right), (z,t)\right\rangle\right| \, dzdt=\frac14,
       \]
       hence by Lemma~\ref{lem:NormalOnCube} we have $U^*\xi=\pm |\xi|e_j$ for some $j=1,\dots,2n$.
       \\
       To prove $(ii)$ we may apply the change of variable $q=p\cdot\delta_r(\Phi_U(z,t))$:
        \begin{equation}\label{eq:Lemma2.4}
        \begin{aligned}
        \int_{\C_{r,U}(p)}|\langle\xi,\pi_H(p^{-1}\cdot q)\rangle|d \leb^{2n+1}(q)&= \int_\C|\langle \xi, \pi_H(p^{-1}\cdot p\cdot \delta_r (Uz,t))\rangle| r^{2n+2}\ dzdt\\&=r^{2n+3}\int_\C|\langle \xi, Uz\rangle| dzdt.
        \end{aligned}
        \end{equation}
       The thesis is then a consequence of inequality $(i)$. The analysis of the equality case for $(ii)$ is a direct consequence of \eqref{eq:Lemma2.4} and $(i)$.
    \end{proof}

\begin{lem}\label{lem_auxacl}
    Let $\rho>0$,  $\eta \in \mathbb{S}^{2n-1}$ and let $U \in U(n)$ be such that $Ue_1=\eta$. Then for every $\xi \in \R^{2n}$ with $\xi \neq 0$ one has the implication
    \[
    \left|\frac{\xi}{|\xi|}-\eta\right|\leq \rho\quad \Rightarrow\quad  \int_\C |\langle \xi, Uz \rangle|dzdt \geq \left( \frac{1}{4}-\frac{\sqrt{2n}}{2}\rho \right) |\xi|.
    \]
\end{lem}
\begin{proof}
Since $U \in U(n)$, we have
$
|U^*\tfrac{\xi}{|\xi|}-e_1|=|\tfrac{\xi}{|\xi|}-Ue_1| \leq \rho.
$
 Moreover, for any $z\in \C_H$ one has $|z| \leq \frac{\sqrt{2n}}{2}$ and $\int_\C |\langle e_1,z\rangle|dzdt =1/4$, hence
\begin{align*}
        \int_\C |\langle \xi,Uz\rangle|dzdt&=  |\xi|\int_\C |\langle U^*\tfrac{\xi}{|\xi|},z\rangle|dzdt\ \geq\ |\xi| \int_\C (|\langle e_1,z\rangle|-|\langle U^*\tfrac{\xi}{|\xi|}-e_1,z\rangle|)dzdt\\&
         \geq \frac{1}{4}|\xi|-|\xi||U^*\tfrac{\xi}{|\xi|}-e_1|\int_\C|z|dzdt\ 
         \geq\ \frac{1}{4}|\xi|-|\xi|\rho\frac{\sqrt{2n}}{2},
    \end{align*}
which is enough to conclude.  
\end{proof}

The next result is a standard consequence of the Taylor formula on step-two Carnot groups, see \cite[Section 4]{gcc10}.

\begin{lem}\label{lem_taylor}
    Let $\Omega \se \H^n$ be an open set and $u \in L^1_{\loc}(\Omega)$. Assume that $\Omega\setminus \J_u$ is open and that $u \in C^\infty(\Omega \setminus \J_u)$ and let $K  \se \Omega \setminus \J_u$ be a compact set. Then there exist constants $C_K>0$, $\delta_K>0$ and a function $R_{p,U}\colon \R\times\C\to \R$ such that the following holds. For every $p \in K$, $U \in U(n)$, $r\in (0,\delta_K)$ and $(z,t) \in \C$ one has that $T_{p,r,U}(z,t)\in \Omega\setminus \J_u$  and 
    \[
    u(T_{p,r,U}(z,t))=u(p)+r \langle \nabla_\H u(p),Uz\rangle+r^2 R_{p,U}(r,z,t), 
    \]
    with
    \[
    \sup\{|R_{p,U}(r,z,t)|:p\in K, U\in U(n), (z,t)\in \C, r\in [0,\delta_K)\} \leq C_K.
    \]
    Consequently, for all such $p\in K$, $U\in U(n)$ and all $r\in [0,\delta_K)$
\begin{equation}\label{eq:stimamediagradiente}
  \left|  \fint_{\C_{r,U}(p)} \left| u-(u)_{\C_{r,U}(p)}\right|d\leb^{2n+1} -r \int_{\C}|\langle \nabla_\H u(p),Uz\rangle|dzdt \right| \leq 2C_Kr^2.
    \end{equation}
\end{lem}

\begin{lem}\label{lem_salti}
    For $\mu \in \mathbb{S}^{2n-1}$ and $\tau,a,b \in \R$ define
    \begin{align*}
        &    E_{\mu,\tau} \ceq \{ (z,t) \in \C:\langle \mu,z \rangle>\tau\}  \\
        &   J_{a,b,\mu,\tau}(z,t) \ceq a \chi_{E_{\mu,\tau}}(z,t)+b \chi_{\C \setminus E_{\mu,\tau}}(z,t).
    \end{align*}
    Then
    \begin{equation}\label{eq_addclaim}
    \int_\C |J_{a,b,\mu,\tau}-(J_{a,b,\mu,\tau})_\C|d\leb^{2n+1} \leq \frac{1}{2}|D_\H J_{a,b,\mu,\tau}|(\C).
    \end{equation}
    Moreover, if $\mu=e_1$ and $|\tau|\leq 1/2$, then
    \begin{equation}\label{eq_eqaux}
    \int_\C |J_{a,b,e_1,\tau}-(J_{a,b,e_1,\tau})_\C|d\leb^{2n+1}=\frac{1}{2}(1-4\tau^2)|a-b|.
    \end{equation}
\end{lem}
\begin{proof}
Let $\lambda=\leb^{2n+1}(E_{\mu,\tau})$. Since $\leb^{2n}(\C_H)=\leb^1((-1/2,1/2))=1$, we have
\[
(J_{a,b,\mu,\tau})_\C=\lambda a+(1-\lambda)b,
\]
and therefore
\begin{equation}\label{eq_auxaux}
\int_\C |J_{a,b,\mu,\tau}-(J_{a,b,\mu,\tau})_\C|d\leb^{2n+1}=2\lambda(1-\lambda)|a-b|.
\end{equation}
Moreover,
\[
|D_\H J_{a,b,\mu,\tau}|(\C)=|a-b|P(E_{\mu,\tau};\C)=|a-b|P(\pi_H(E_{\mu,\tau});\mathcal C_H)
\]
where $P(E_{\mu,\tau};\C)$ (resp.\ $P(\pi_H(E_{\mu,\tau});\mathcal C_H)$) is the Euclidean perimeter in $\R^{2n+1}$ (resp. in $\R^{2n}$) and the first equality comes from \cite[Proposition 2.10]{scv14}. By the sharp relative isoperimetric inequality (see \cite[Appendix]{abbf16}) one has
$
4\lambda (1-\lambda) \leq P(\pi_H(E_{\mu,\tau});\C_H),
$
which proves \eqref{eq_addclaim}. We obtain \eqref{eq_eqaux} from \eqref{eq_auxaux} upon observing that if $\mu=e_1$ and $|\tau|\leq 1/2$ then $\lambda=1/2-\tau$.
\end{proof}

\section{Good functions with bounded variation}\label{sec_goodSBV}
In this section we introduce {\em good} $\sbv_\H$ functions (Definition~\ref{def_goodsbv}). The main result is Lemma~\ref{lem_comb}, where we collect the properties of $\sbv_\H$ relevant for our purposes; before that, however, we need the following quite standard result on $\sbv_\H$ functions.

\begin{lem}\label{lem:jumplocalcomparison}
Let $\Omega\subset \Hn$ be open, let $u\in\sbv_\H(\Omega)$ and let $H\subset\mathcal J_u$ be compact.  For every $\eta>0$ there exists a compact set $K\subset H$ such that
\[
 |D_\H^ju|(H\setminus K)<\eta
\]
and, for every $R>0$,
\begin{equation} \label{eq:limiteborel}
 \lim_{r\to 0} \sup\left\{ \frac{|D_\H^au|(p\cdot\delta_rA)}{r^{2n+1}}:p\in K,A\subset B(0,R)\text{ Borel} \right\}=0.
\end{equation}
Consequently, for every $\varepsilon>0$ and $R>0$ there exists $r_{\ve,R}>0$ such that
\begin{equation}\label{eq:stimasalto}
 |D_\H u|(p\cdot\delta_rA) \le |D_\H^ju|(p\cdot\delta_rA)+\varepsilon r^{2n+1}
\end{equation}
whenever $p\in K$, $0<r<r_{\varepsilon,R}$ and $A\subset B(0,R)$ is Borel.
\end{lem}
\begin{proof}
    Since $\leb ^{2n+1}(\J_u)=0$, also $|D_\H^a u|(H)=0$. We first show that 
    \begin{equation}\label{eq:limitepuntuale}
      \lim_{r\to 0}\frac{|D_\H^a u|(B(p,r))}{r^{2n+1}}=0, \quad \text{for $\hau^{2n+1}$-a.e.\ $p\in H$.}
    \end{equation}
    Define for every $a>0$ the set
    \[
    E_a \ceq \left\{p\in H: \limsup_{r\to 0}\frac{|D_\H^a u|(B(p,r))}{r^{2n+1}}>a\right\}.
    \]
    Notice that $E_a\subset E_b$ if $a>b$. If $E_a$ is empty for every $a>0$, then \eqref{eq:limitepuntuale} is proved for every $p\in H$. Otherwise we can assume that for every sufficiently small $a>0$, $E_a$ is non-empty. Let $\delta>0$ be sufficiently small so that $H_\delta\coloneqq\{x\in \Hn: d(x,H)<\delta\}\subset \Omega$. Choose for each $p\in E_a$ some $r_p\in (0,\delta)$ with the property that
    \[
    |D^a_\H u|(B(p,r_p))> ar_p^{2n+1}.
    \]
    The $5r$-covering Lemma provides a pairwise disjoint countable subfamily $\{B(p_i,r_{p_i}):i\in \mathbb N\}$ such that $E_a\subset \bigcup _{i\in \mathbb N} B(p_i,5r_{p_i})$. Hence
    \[
    \hau^{2n+1}_{10\delta} (E_a)\leq 10^{2n+1}\sum_{i\in \mathbb N} r_{p_i}^{2n+1}\leq \frac{10^{2n+1}}a\sum_{i\in \mathbb N} |D^a_\H u|(B(p_i,r_{p_i}))\leq \frac{10^{2n+1}}a |D^a_\H u|(H_\delta).
    \]
    Taking into account that $H$ is compact and that $|D^a_\H u|(H)=0$, letting $\delta\to 0$ we infer
    \[
    \hau^{2n+1}(E_a)=0, \quad \forall a>0.
    \]
    Hence 
    \[
    \hau^{2n+1}\left(\bigcup_{a\in \mathbb Q,a>0} E_a\right)=0.
    \]
    This proves \eqref{eq:limitepuntuale}. Since, by \cite[Theorem 1.7]{dv}, $ |D^j_\H u|\ll \hau^{2n+1}\res \J_u$, the limit \eqref{eq:limitepuntuale} also holds for $|D^j_\H u|$-a.e.\ $p\in H$.  For every $k\in \mathbb N$ define $h_k\colon \Omega \to \R$ by letting
    \[
    h_k(q) \ceq 2^{k(2n+1)}|D^a_\H u|(B(q,2^{-k})\cap \Omega).
    \]
    Since $\{p\mapsto |D_\H^a u|(B(p,r))\}$ is continuous, each $h_k$ is continuous and by \eqref{eq:limitepuntuale} we have
    \[
    \lim_{k \to +\infty} h_k(p)=0, \quad \text{ for $|D^j_\H u|$-a.e.\ $p\in H$.}
    \]
     By Egorov's Theorem, we can find $K\subset H$ such that $|D^j_\H u|(H\setminus K)<\eta$ and 
    \[
    \lim_{k \to +\infty}\ \sup_{p\in K} h_k(p)=0.
    \]
    Hence, if $2^{-(k+1)}<r\leq 2^{-k}$, then
    \[
    \frac{|D^a_\H u|(B(p,r))}{r^{2n+1}}\leq \frac{|D^a_\H u|(B(p,2^{-k}))}{2^{-(k+1)(2n+1)}}=2^{2n+1} h_k(p).
    \]
    This shows that
    \begin{equation}\label{eq:limiteuniforme}
    \lim_{r\to 0}\ \sup_{p\in K} \frac{|D^a_\H u|(B(p,r))}{r^{2n+1}}=0.
    \end{equation}
    Now let $R>0$ and let $A$ be a Borel subset of  $B(0,R)$. Then $p\cdot\delta_r A\subseteq B(p, rR)$. Therefore if $p\in K$, there exists $r_0>0$ such that $p\cdot \delta_r A\subseteq \Omega$ for all $r\in (0,r_0)$. Then the simple inequality
    \[
    \frac{|D^a_\H u|(p\cdot \delta_r A)}{r^{2n+1}}\leq R^{2n+1}\frac{|D^a_\H u|(B(p,rR))}{(rR)^{2n+1}},
    \]
    together with \eqref{eq:limiteuniforme} concludes the proof of \eqref{eq:limiteborel}. Inequality \eqref{eq:stimasalto} is an immediate consequence of \eqref{eq:limiteborel} and of the decomposition $   |D_\H u|=|D_\H^a u|+|D^j_\H u|$.
    \end{proof}

Inspired by the properties of the approximating functions in \cite[Theorem 1.4]{sbvx}, in what follows we will use a special class of $\sbv_\H$ functions, that we call \emph{good} $\sbv_\H$.

\begin{defi}\label{def_goodsbv}
    Let $\Omega \se \H^n$ be an open set. We say that a function $u \in \sbv_\H(\Omega)$ is a \emph{good} $\sbv_\H(\Omega)$ function if it satisfies the following conditions:
    \begin{enumerate}
        \item[(i)] $\J_u$ is compact and contained in a $C^1_\H$ hypersurface,
        \item[(ii)] $u \in C^\infty(\Omega \setminus \J_u)$.
    \end{enumerate}
\end{defi}

In the following lemma, which will be pivotal for the proof of our main result, we make use of the following convenient notation: given an open set $\Omega \se \Hn$, $u \in \sbv_\H(\Omega)$, $q \in \J_u$ and a horizontal normal $\nu_u(q)$ we define the function
\[
u_q^0 \ceq u^+(q)\chi_{H_{\nu_u(q)}^+}+u^-(q)\chi_{H_{\nu_u(q)}^-},
\]
where, for every $\nu \in \mathbb S^{2n-1}$, we define
\[
H_{\nu}^+ \ceq \{(z,t) \in \Hn: \langle \nu,z \rangle >0 \}, \qquad H_{\nu}^- \ceq \{(z,t) \in \Hn: \langle \nu,z \rangle <0 \}.
\]

\begin{lem}\label{lem_comb}
     Let $\Omega \se \Hn$ be an open set and $u$ be a \emph{good} $\sbv_\H(\Omega)$ function. For every $\eta>0$ and $0<\sigma<1$ there exist pairwise disjoint compact sets
    \[
    K_1,\dots,K_m \se \J_u, 
    \]
and a suitable orientation of $\nu_u$ such that, letting $K \ceq \bigcup_{i=1}^m K_i$, the following facts hold.
    \begin{enumerate}
    \item[(i)]  The maps $u^+,u^-,\nu_u$  are continuous on $K$.
    \item[(ii)] $|D_\H^ju|(\J_u \setminus K)<\eta$ and \eqref{eq:limiteborel} holds on $K$.
     \item[(iii)] For every $i=1,\dots,m$, there exist a rotation $U_i \in U(n)$, an open set $A_i\subset \W\ceq \{x_1=0\}$ and a  function $\varphi_i\colon A_i\to \R$, such that $K_i \se \Phi_{U_i}(\Gamma_i)$ where $\Gamma_i$ is the  $C^1_\H$-intrinsic graph of $\varphi_i$, namely
        \[
        \Gamma_i \ceq \{F_i(y)\ceq y \cdot \vp_i(y)e_1: y \in A_i \se \W \}
        \]
        and 
        \[
        \sup_{p \in K_i} |U_i^*\nu_u(p)-e_1|<\sigma.
        \]
    \item[(iv)]  If we define for every $p \in K_i$  and every $r>0$ the sets
        \begin{align*}
      &  \C^{+}_{r,U_i}(p) \ceq p \cdot (\delta_r \Phi_{U_i}(\{(z,t)\in \C:\langle U_i^*\nu_u(p),z \rangle>0\})),\\& \C^{-}_{r,U_i}(p) \ceq p\cdot (\delta_r \Phi_{U_i}(\{(z,t)\in \C:\langle U_i^*\nu_u(p),z \rangle<0\})),
      \end{align*}
        then
        \[
     \alpha(r) \ceq \sup_{\pm}   \sup_{1 \leq i \leq m}\sup_{p \in K_i}\fint_{\C^{\pm}_{r,U_i}(p) } |u-u^\pm(p)|d\leb^{2n+1} \xrightarrow{r \to 0}0.
        \]
         \item[(v)] For every bounded Borel set $B\subset\H^n$ such that $\leb^{2n+1}(B)>0$ and $\leb^{2n+1}(\pp B)=0$ one has
\[
\sup_{q\in K}\fint_B\left|u(q\cdot\delta_r x)-u_q^0(x)\right|dx\xrightarrow{r \to 0}0.
\]
\item[(vi)] Let $(q_k)_{k \in \N}$ be a sequence of points in $K_i$ for some $i=1,\dots,m$, with $q_k\to q\in K_i$, let $(r_k)_{k \in \N}$ be a sequence in $(0,+\infty)$ with $r_k\to 0$ and let $(\C_k)_{k \in \N}$ be a sequence of Heisenberg $1$-cubes converging in $L^1$ to some unit Heisenberg $1$-cube $\C_\infty$. Then
\[
\liminf_{k\to\infty} r_k^{-(2n+1)}|D_\H^j u|\bigl(q_k\cdot\delta_{r_k}(\C_k)\bigr)\geq|u^+(q)-u^-(q)|\,P_\H(H_{\nu_u(q)}^+,\C_\infty).
\]

\item[(vii)] For every $\ve>0$ there exists $r_\ve>0$ such that, if $0<r<r_\ve$, $U\in U(n)$, $p\in\H^n$ and $\C_{r,U}(p)\cap K\neq\emptyset$, then
\[
r^{2n+1}\fint_{\C_{r,U}(p)}|u-(u)_{\C_{r,U}(p)}|d\leb^{2n+1}\leq \frac12 |D_\H^j u|(\J_u\cap \C_{r,U}(p))+\varepsilon r^{2n+1}.
\]

\item[(viii)] There exists $r_0>0$  such that for every $r\in(0,r_0)$ and every $i=1,\dots, m$ we can find $N_i(r)\in \mathbb N$ and points $p_{i,1}^r,\dots, p_{i,N_i(r)}^r \in K_i$ such that, setting $\C_{i,h}^r\coloneqq \C_{r,U_i}(p_{i,h}^r)$, the family
\[
\mathcal F_r\ceq\bigcup_{i=1}^m\mathcal F_r^i,\qquad\text{where }\mathcal F_r^i \ceq  \{\C_{i,h}^r: h=1,\dots, N_i(r)\} 
\]
is pairwise disjoint, compactly contained in $\Omega$ and the following properties hold:
\begin{itemize}
\item[(a)] $\displaystyle\sup_{0<r<r_0} N_i(r)r^{2n+1}<+\infty;
$
\item[(b)]$\displaystyle\sum_{h=1}^{N_i(r)}
 |u^+(p_{i,h}^r)-u^-(p_{i,h}^r)| r^{2n+1}
 \ge
 (1-\sigma)|D_{\mathbb H}^j u|(K_i)$.
 \end{itemize}

   \end{enumerate}

\end{lem}
\begin{proof}
Properties $(i)$ and the first part of $(ii)$ follow in the same fashion as in \cite[Lemma 4.2]{sbvx}, while the validity of \eqref{eq:limiteborel} is a consequence of Lemma~\ref{lem:jumplocalcomparison}. Since  $K\subset\J_u$ is contained in a $C^1_\H$ hypersurface, then $K$ can be covered by compact sets $K_1,\dots, K_m$ so that property $(iii)$ holds, see \cite{fssc01}. Here and in the sequel, we tacitly agree that $K$ can be replaced by a slightly smaller compact subset and that each $K_i$ can be made smaller and or further subdivided, so that also the integer $m$ can increase. 
With this {caveat}, properties $(iv)$ and $(v)$ follow by \cite[Proposition 2.26]{dv} and by Egorov's theorem.

We now prove $(vi)$. Up to passing to a subsequence, we can assume that the $\liminf$ in the statement is indeed a limit; if infinite, there is nothing to prove. Otherwise, since $\C_k$ is converging to $\C_\infty$, then for every open set $\C_\infty'\subset \subset \C_\infty$  one has $\C_\infty'\subset \C_k\subset B(0,\mathfrak b)$ (recall~\eqref{eq:cube-ball}) for every $k\in \N$ sufficiently large. 
Denote by $u_k\colon B(0,\mathfrak b)\to \R$ the function defined by
\[
u_k(x)=u(q_k\cdot \delta_{r_k}(x)).
\]
Notice that $u_k$ is well defined  for every sufficiently large $k$ because $K_i$ is compactly contained in $\Omega$. Define also the functions $u_{q_k}^0\colon \H^n\to \R$ and $u_q^0\colon \H^n\to \R$ by setting
\[
u_{q_k}^0=u^+(q_k)\chi_{H^+_{\nu_u(q_k)}}+u^-(q_k)\chi_{H^-_{\nu_u(q_k)}} \qquad \text{and} \qquad u_q^0=u^+(q)\chi_{H^+_{\nu_u(q)}}+u^-(q)\chi_{H^-_{\nu_u(q)}}.
\]
We claim that $u_k$ converges to $u_q^0$ in $L^1(\C_\infty')$. 
Clearly we have
\begin{equation}\label{eq_esamiorali}
\|u_k-u_q^0\|_{L^1(\C_\infty')}\leq \|u_k-u_{q_k}^0\|_{L^1(\C_\infty')}+\|u_{q_k}^0-u_q^0\|_{L^1(\C_\infty')}.
\end{equation}
We know by $(i)$ that $u^\pm(q_k)$ is converging to $u^\pm(q)$ as $k\to \infty$. Moreover, since $\nu_u(q_k)$ converges to $\nu_u(q)$, also the halfspaces $H^\pm_{\nu_u(q_k)}$ converge in $L^1_{\rm loc}(\H^n)$ to $H^\pm_{\nu_u(q)}$. Hence $u_{q_k}^0\to u_q^0$ in $L^1(\C_\infty)$ as $k\to \infty$. In order to prove that $u_k\to u_{q_k}^0$ in $L^1(\C_\infty')$ it is enough to apply $(v)$ with $B=\C_\infty'$. Then, the lower semicontinuity of the total variation gives 
\[
|D_\H u_q^0|(\C_\infty')\leq \liminf_{k \to +\infty} |D_\H u_k|(\C_\infty'),
\]
whence, with direct calculation
\begin{equation}\label{eq:semicontinuita}
|u^+(q)-u^-(q)|P_\H(H_{\nu_u(q)}^+; \C_\infty')\leq \liminf_{k \to +\infty} r_k^{-2n-1}|D_\H u|(q_k\cdot \delta_{r_k}(\C_\infty')).
\end{equation}
By Lemma~\ref{lem:jumplocalcomparison} and since \eqref{eq:limiteborel} (and thus~\eqref{eq:stimasalto}) holds on $K$, for any $\ve>0$, we can find $r_\ve>0$ such that
\[
|D_\H u| (x\cdot \delta_r(\C_k))\leq  |D^j_\H u|( x\cdot \delta_r (\C_k))+\ve r^{2n+1}, \qquad \forall r\in (0,r_\ve), \forall x\in K_i, \forall k\in \N. 
\]
Using the fact that $\C'_\infty\subset \C_k $ we can then apply the previous inequality to the right-hand side of \eqref{eq:semicontinuita} to get
\[
    |u^+(q)-u^-(q)|P_\H(H_{\nu_u(q)}^+; \C_\infty')\leq \lim_{k \to +\infty} r_k^{-2n-1}|D_\H^j u|(q_k\cdot \delta_{r_k}(\C_k))+\ve.
\]
The arbitrariness of $\C_\infty'$ and of $\varepsilon>0$ allows us to conclude the proof.

Let us prove $(vii)$: assume by contradiction that the statement fails. Then there exist $\ve_0>0$, radii $r_k \to 0$,  rotations $U_k\in U(n)$, points $p_k\in\H^n$ and Heisenberg $r_k$-cubes $\C_k \ceq \C_{r_k,U_k}(p_k)$ such that $\C_k\cap K\neq\emptyset$ 
and
\begin{equation}\label{eq_contradiction}
r_k^{2n+1} \fint_{\C_k}|u-(u)_{\C_k}|\,d\leb^{2n+1}>\frac12 |D_\H^j u|(\J_u\cap \C_k)+\varepsilon_0 r_k^{2n+1}.
\end{equation}
Choose $q_k\in \C_k\cap K$. 
Up to a subsequence every $q_k$ is contained in $K_{\bar\imath}$ for some fixed $\bar\imath$ and $q_k\to q\in K_{\bar\imath}$. 
For any large enough $k \in \N$ we let $u_k\colon B(0,\mathfrak b) \to \R$ by defining 
\[
u_k(x) \ceq u(q_k\cdot\delta_{r_k}x),
\qquad
\widetilde \C_k \ceq \delta_{1/r_k}(q_k^{-1} \cdot \C_k).
\]
Since $q_k\in \C_k$, the sets $\widetilde \C_k$ belong to a compact family of Heisenberg 1-cubes: their centers remain bounded and the rotations $U_k$ stay in a compact group. Hence, up to a subsequence, $\widetilde \C_k\to \C_\infty$ for some Heisenberg $1$-cube $\C_\infty$. Using $(v)$ and arguing as in~\eqref{eq_esamiorali} one finds
\[
u_k\xrightarrow{k \to +\infty} u_q^0\ceq u^+(q)\chi_{H_{\nu_u(q)}^+}+u^-(q)\chi_{H_{\nu_u(q)}^-}\qquad\text{in }L^1_{\loc}(B(0,\mathfrak b))
\]
whence
\begin{equation}\label{eq_osc}
\lim_{k \to +\infty}\fint_{\widetilde \C_k}|u_k-(u_k)_{\widetilde \C_k}|d\leb^{2n+1}=\fint_{\C_\infty}|u_q^0-(u_q^0)_{\C_\infty}|d\leb^{2n+1}.
\end{equation}
 Moreover, by $(vi)$,
\begin{equation}\label{eq_liminf}
\liminf_{k\to\infty} r_k^{-(2n+1)}|D_\H^j u|(\J_u\cap \C_k)\geq |u^+(q)-u^-(q)|P_\H(H_{\nu_u(q)}^+;\C_\infty).
\end{equation}
Dividing \eqref{eq_contradiction} by $r_k^{2n+1}$ and using \eqref{eq_osc} and \eqref{eq_liminf}, we obtain
\[
\fint_{\C_\infty}|u_q^0-(u_q^0)_{\C_\infty}|d\leb^{2n+1}\geq\frac12|u^+(q)-u^-(q)|P_\H(H_{\nu_u(q)}^+;\C_\infty)+\varepsilon_0.
\]
The above inequality contradicts Lemma~\ref{lem_salti} (when applied to $\C_\infty$), therefore $(vii)$ is proved. 

We now prove $(viii)$. Fix $i \in \{ 1,\dots,m \}$. We work in the coordinates obtained by applying the inverse $\Phi_{U_i}^{-1}$. More precisely, after replacing $\Omega, u,K_i$ with $\Phi_{U_i}^{-1}(\Omega), u\circ\Phi_{U_i},\Phi_{U_i}^{-1}(K_i)$, respectively, and suppressing   for the sake of brevity the index $i$ from $K_i,\Gamma_i, A_i, F_i,\varphi_i$ (recall statement $(iii)$), we
may assume that $U_i=\Id$ and $K\subset\Gamma$. Set
$ A_K \ceq F^{-1}(K)\subset\W$. 
The map $F\colon A\to\Gamma$ is a homeomorphism onto the intrinsic graph, hence $A_K$ is compact. Define
\[
 J(x)\coloneqq|u^+(x)-u^-(x)|,
 \qquad \forall x\in K.
\]
Using the area formula for $C^1_{\mathbb H}$ intrinsic graphs (see \cite[Theorem 6.5]{fssc01}) and  \cite[Theorem 1.7]{dv} one gets
\[
 |D_{\mathbb H}^j u|(K) = \int_{A_K} J(F(y)) \frac{1}{|\langle\nu_u(F(y)),e_1\rangle|}d\leb_\W(y),
\]
where $\leb_\W$ denotes the Lebesgue measure on $\W\equiv\R^{2n}$. Since $|\nu_u-e_1|<\sigma<1$ on $K$ and the orientation has been chosen so that
$\langle\nu_u,e_1\rangle>0$ on $K$, we have
\[
 1\leq \frac{1}{|\langle\nu_u(F(y)),e_1\rangle|} \leq \frac1{1-\tfrac{\sigma^2}2} \qquad\text{for all }y\in A_K
\]
because $|\nu_u|=|e_1|=1$, hence $\langle\nu_u,e_1\rangle=1-\frac12|\nu_u-e_1|^2 > 1-\frac{\sigma^2}{2}$.
Therefore 
\begin{equation}\label{eq_aaaa} 
\int_{A_K}J(F(y))d\leb_\W(y) \geq \left(1-\frac{\sigma^2}2\right)|D_{\mathbb H}^j u|(K).
\end{equation}
We write points of $\W$ as $ y=(\widehat z,\tau)\in\R^{2n-1}\times\R\equiv\W=\{x_1=0\}$.
Define for any $r>0$ the set $P_r\subseteq \mathbb W$ by setting
$
 P_{ r}\ceq[0, r)^{2n-1}\times[0,r^2)
$
and observe  that
$
 \leb_\W(P_{r})=r^{2n+1}.
$
Choose $\lambda \in \R$ such that 
\[
\sqrt{\frac{2+\frac{\sigma}{2n+1}}{2-\frac{\sigma}{2n+1}}}<\lambda <\left(\frac{1-\tfrac{\sigma^2}{2}}{1-\sigma}\right)^\frac{1}{2n+1}.
\]
Such a  choice  is always possible because for $0<\sigma<1$ one can check\footnote{The functions $g(\sigma)=\sqrt{\frac{2+\frac{\sigma}{2n+1}}{2-\frac{\sigma}{2n+1}}}$ and $h(\sigma)=\left(\frac{1-\sigma^2/2}{1-\sigma}\right)^\frac{1}{2n+1}$ are such that $g(0)=h(0)$ and $g'(\sigma)<\tfrac1{2n+1}\leq h'(\sigma)$ for any $\sigma\in(0,1)$.} that
\[
\sqrt{\frac{2+\frac{\sigma}{2n+1}}{2-\frac{\sigma}{2n+1}}}<\left(\frac{1-\tfrac{\sigma^2}{2}}{1-\sigma}\right)^\frac{1}{2n+1}.
\]
For every $r>0$ and $\theta\in P_{\lambda r}$, we define the lattice
\[
 \Lambda_{r,\theta} \ceq \theta + \{(\lambda r \widehat k,\lambda^2r^2k_t): \widehat k\in\Z^{2n-1}, k_t\in\Z\}\se\W.
\]
By Fubini's theorem,
\[
 \frac{1}{\leb_\W(P_{\lambda r})} \int_{P_{\lambda r}} \sum_{y\in\Lambda_{r,\theta}\cap A} J(F(y))\chi_{A_K}(y) \leb_\W(P_{\lambda r})d\theta =\int_{A_K}J(F(y))d\leb_\W(y).
\]
Hence there exists $\theta_r\in P_{\lambda r}$ such that
\[
 \sum_{y\in\Lambda_{r,\theta_r}\cap A_K} J(F(y))\leb_\W(P_{\lambda r}) \geq \int_{A_K}J(F(y))d\leb_\W(y).
\]
Since $A_K$ is compact and the lattice $\Lambda_{r,\theta_r}$ is discrete, $\Lambda_{r,\theta_r}\cap A_K$ is finite. We write 
\[
 \Lambda_{r,\theta_r}\cap A_K =\{y_1^r,\ldots,y_{N(r)}^r\}, \qquad p_h^r \ceq F(y_h^r).
\]
Then for $r>0$
\[
 \sum_{h=1}^{N(r)}J(p_h^r)r^{2n+1} = \lambda^{-(2n+1)} \sum_{h=1}^{N(r)} J(F(y_h^r))\leb_\W(P_{\lambda r})  \geq \lambda^{-(2n+1)} \int_{A_K}J(F(y))d\leb_\W(y).
\]
By the definition of $\lambda$ we have
\[
\lambda^{-(2n+1)} \geq \frac{1-\sigma}{1-\tfrac{\sigma^2}2}
\]
and, combining the previous considerations with \eqref{eq_aaaa}, we get
\[
 \sum_{h=1}^{N(r)} |u^+(p_h^r)-u^-(p_h^r)| r^{2n+1} \geq (1-\sigma)|D_{\H}^j u|(K),
\]
which is statement $(b)$.

The cardinality estimate $(a)$ follows  from the lattice structure. For every
$
y\in\Lambda_{r,\theta_r}\cap A_K
$
define the half-open lattice cell
$
Q_{r,y} \ceq y+P_{\lambda r}.
$
The sets $Q_{r,y}$ are pairwise disjoint, since the family
$
\{y+P_{\lambda r}:y\in\Lambda_{r,\theta_r}\}
$
is a paving of $\W$. Moreover,
\[
\mathcal L_\W(Q_{r,y})=\mathcal L_\W(P_{\lambda r})=\lambda^{2n+1}r^{2n+1}.
\]
Since $A_K$ is compact, the set
$
B_K \ceq A_K+\overline{P_\lambda}
$
is compact and has finite $\mathcal L_\W$-measure. For $0<r<1$ one has $P_{\lambda r}\subseteq P_\lambda$,  hence
\[
Q_{r,y}\subseteq B_K \qquad \text{for every } y\in\Lambda_{r,\theta_r}\cap A_K.
\]
It follows that
\[
N(r)\lambda^{2n+1}r^{2n+1}=\sum_{y\in\Lambda_{r,\theta_r}\cap A_K}\mathcal L_\W(Q_{r,y})\leq\mathcal L_\W(B_K).
\]
Consequently,
\[
N(r)r^{2n+1}\leq\lambda^{-(2n+1)}\mathcal L_\W(B_K)\eqqcolon C_K
\]
for every $0<r<1$. After possibly decreasing $r_0$, this proves statement $(a)$.

We aim to prove that the Heisenberg cubes $\C_h^r\coloneqq \C_{r,\Id}(p_h^r)$ are pairwise disjoint for every $r\in (0,r_0)$.\footnote{It is enough to prove the claim in the coordinates determined by $\Phi^{-1}_{U_i}$. Indeed, $\Phi_{U_i}$ is a bijective group automorphism and commutes with the homogeneous dilations. Therefore,
\[
\Phi_{U_i}\bigl(\C_{r,\Id}(p_h^r)\bigr)=
\Phi_{U_i}\bigl(p_h^r\cdot\delta_r(\C)\bigr)=
\Phi_{U_i}(p_h^r)\cdot
\delta_r\bigl(\Phi_{U_i}(\C)\bigr)=
\C_{r,U_i}\bigl(\Phi_{U_i}(p_h^r)\bigr).
\]
Consequently, the Heisenberg $r$-cubes $\C_{r,\Id}(p_h^r)$ are pairwise disjoint if and only if the corresponding Heisenberg $r$-cubes $\C_{r,U_i}\bigl(\Phi_{U_i}(p_h^r)\bigr)$ are pairwise disjoint. Since $\Phi_{U_i}(p_h^r)\in K_i$, the latter are precisely the Heisenberg cubes required in the statement for the index $i$.}
Assume by contradiction that there exists $i\neq j$ such that $\C_i^r\cap \C_j^r\neq\emptyset$. Then there exist
$
 (\xi_1,\widehat\xi,s), (\xi_1',\widehat\xi',s') \in \C\se \H^n \equiv \R \times \R^{2n-1} \times \R
$
such that
\begin{equation}\label{eq:intersectionequality}
p_i^r\cdot\delta_r(\xi_1,\widehat\xi,s) = p_j^r\cdot\delta_r(\xi_1',\widehat\xi',s').
\end{equation}
Write 
$
y_i^r=(\widehat z_i,\tau_i) \in \W \equiv \R^{2n-1}\times \R, \ y_j^r=(\widehat z_j,\tau_j)\in \W \equiv \R^{2n-1}\times \R
$
and recall that
\[
p_i^r=F(y_i^r)=y_i^r\cdot\vp(y_i^r)e_1, \qquad p_j^r=F(y_j^r)=y_j^r\cdot\vp(y_j^r)e_1.
\]
Then, writing each component of equality \eqref{eq:intersectionequality} and denoting by $(z_\ell)_{n+1}$ the $n+1$ component of $(0,\widehat z_\ell)$, one gets
\begin{equation}\label{eq_sistema}
\begin{cases}
    \vp(y_i^r)+r \xi_1=\vp(y_j^r)+r \xi_1',\\
    \widehat{z_i}+r\widehat{\xi}=\widehat{z_j}+r\widehat{\xi}',\\
    \tau_i+2(z_i)_{n+1}\vp(y_i^r)+r^2s+2\omega\bigg( \big(\vp(y_i^r),\widehat z_i\big),\big( r\xi_1,r\widehat\xi\big)\bigg)=\\
    \qquad=\tau_j+2(z_j)_{n+1}\vp(y_j^r)+r^2s'+2\omega\bigg( \big(\vp(y_j^r),\widehat z_j\big),\big( r\xi_1',r\widehat\xi'\big)\bigg).
\end{cases}
\end{equation}
Since $\widehat\xi,\widehat\xi'\in (-1/2,1/2)^{2n-1}$ belong to the unit cube, we get from the second equation of \eqref{eq_sistema}
\[
  |\widehat z_i-\widehat z_j|_{\infty}\leq r.
\]
On the other hand, both $y_i^r$ and $y_j^r$ belong to the lattice $\Lambda_{r,\theta_r}$, whose spacing in the $\widehat z$ variables is $\lambda r$, with $\lambda>1$. Therefore, if $\widehat z_i\neq\widehat z_j$, then
\[
 |\widehat z_i-\widehat z_j|_{\infty}\geq\lambda r>r,
\]
a contradiction. Hence $\widehat z_i=\widehat z_j$ and, again from the second equation of \eqref{eq_sistema}, we infer that $\widehat \xi=\widehat \xi'$. From the last equation of \eqref{eq_sistema} we obtain, since $\widehat z_i=\widehat z_j$ and $\widehat \xi =\widehat \xi'$, that
\[
\tau_i-\tau_j=r^2(s'-s)-2r(\vp(y_i^r)-\vp(y_j^r))\omega(e_1,(0,\widehat\xi))
\]
Since $s,s' \in (-1/2,1/2), \widehat \xi=\widehat \xi' \in (-1/2,1/2)^{2n-1}$ we have that $|\omega(e_1,(0,\widehat\xi))|\leq 1/2$ and $|s-s'| \leq 1$ and, consequently,
\begin{equation}\label{eq_01}
|\tau_i-\tau_j| \leq r^2+r|\vp(y_i^r)-\vp(y_j^r)|.
\end{equation}
Combining the latter with the first equation of \eqref{eq_sistema} and the fact that $\widehat z_i=\widehat z_j$ we infer
\[
|y_i^r-y_j^r|_{\R^{2n}}=|\tau_i-\tau_j| \leq r^2+r|\xi_1-\xi_1'|\leq r^2+r,
\]
and by the  {\em little} $\tfrac 12$-H\"older continuity of $\varphi$ (see~\cite[Theorem 1.3]{ascv06}) one can find $r_0>0$ such that for every $0<r<r_0$
\begin{equation}\label{eq_02bis}
|\vp(y_i^r)-\vp(y_j^r)| \leq \beta |y_i^r-y_j^r|_{\R^{2n}}^\frac{1}{2}=\beta|\tau_i-\tau_j|^\frac{1}{2},
\end{equation}
where $\beta =\beta(\sigma) \ceq \frac{\sigma}{2n+1}$. From \eqref{eq_01} we also obtain 
\[
|\tau_i -\tau_j| \leq r^2+ \frac{\beta r^2}{2}+\frac{|\vp(y_i^r)-\vp(y_j^r)|^2}{2\beta}
\]
which, combined  with \eqref{eq_02bis}, gives
\[
|\tau_i -\tau_j| \leq \left(1+\frac{\beta}{2}   \right)r^2+\beta\frac{|\tau_i-\tau_j|}{2},
\]
that is
\begin{equation}\label{eq_contr}
|\tau_i -\tau_j| \leq \left( \frac{2+\beta}{2-\beta}   \right)r^2.
\end{equation}
Since $y_i^r\neq y_j^r$ and $\widehat z_i=\widehat z_j$, we have $\tau_i\neq \tau_j$. Thus, the vertical spacing of the lattice gives
\[
\lambda^2 r^2 \leq |\tau_i-\tau_j| \leq \left( \frac{2+\beta}{2-\beta}   \right)r^2,
\]
which, recalling that $\beta=\frac{\sigma}{2n+1}$, contradicts the definition of $\lambda$. Therefore the Heisenberg cubes $\C_h^r$ are pairwise disjoint. Finally, since $K\subset \subset\Omega$, up to reducing $r_0$ we may assume that all the Heisenberg cubes $\C_h^r$ are compactly contained in $\Omega$. Repeating the above construction for every $i=1,\ldots,m$ proves the pairwise disjointness within each family $\mathcal F_r^i$. 

It remains only to ensure the disjointness between Heisenberg cubes associated with different compact sets. Since $K_1,\dots,K_m$ are finitely many pairwise disjoint compact sets,
after discarding the empty ones we have
$
d_0 \ceq \min_{i\neq j}\dist(K_i,K_j)>0.
$
By \eqref{eq:cube-ball}, every Heisenberg cube centered at a point of $K_i$ is contained in the ball of radius $\mathfrak b r$ centered at that point. Therefore, after possibly decreasing $r_0$ so that $2 \mathfrak br_0<d_0$, a Heisenberg cube centered in $K_i$ cannot intersect a Heisenberg cube centered in $K_j$ whenever $i\neq j$. Hence the entire family
$
\mathcal F_r \ceq \cup_{i=1}^m\mathcal F_r^i
$
is pairwise disjoint for every $0<r<r_0$, and the proof is concluded.
\end{proof}

The last preliminary result needed for our purposes is a Poincaré inequality. The following theorem in the general setting of Carnot-Carathéodory spaces is well known where John domains are considered in place of the cube $\C$. For a deeper study of Poincaré inequalities we refer the reader to \cite{HajlaszKoskela, MontiMorbidelli} and references therein. 

\begin{teo}\label{teo_poincare}
    Let $\Omega\subseteq \mathbb H^n$ be an open set. There exists a constant $C_P>0$ such that for every $u\in \bv_\H(\Omega)$, every $U \in U(n)$, $r>0$ and  $p \in \Omega$ such that $\C_{r,U}(p) \subset \subset \Omega$ one has
    \[
    r^{2n+1}\fint_{\C_{r,U}(p)}|u-(u)_{\C_{r,U}(p)}|d\leb^{2n+1} \leq  C_P |D_\H u|(\C_{ r,U}(p)).
    \]
 \end{teo}
 \begin{proof}
    The fact that Euclidean cubes are John domains in the Heisenberg groups is a consequence of \cite{Greshnov}\footnote{The author actually shows that such sets are uniform sets which in turn implies the John property.}. In particular, (see e.g.~\cite{MontiMorbidelli}) there exists $C_P>0$ such that for every $\tilde u\in \bv_\H(\C)$ one has
    \begin{equation}\label{eq:poincaremonti}
    \left(\int_{\C}|\tilde u-(\tilde u)_{\C}|^{\frac{2n+2}{2n+1}}d\leb^{2n+1}\right)^{\frac{2n+1}{2n+2}} \leq  C_P |D_\H \tilde u|(\C).
    \end{equation}
    Let $r>0$, $U\in U(n)$ and $p\in \Omega$ be such that $\C_{r,U}(p)\subset \Omega$ and let $u\in \bv_\H(\Omega)$. Define $\tilde u\colon \C\to \R$ by setting
    $
    \tilde u(z,t)\ceq u(p\cdot\delta_r(Uz,t)).
    $
    Then $\tilde u\in \bv_\H(\C)$ and 
    \[
    (\tilde u)_\C=    (u)_{\C_{r,U}(p)}
    \qquad\text{and}\qquad
    |D_\H \tilde u|(\C)=\frac{1}{r^{2n+1}} |D_\H u|(\C_{r,U}(p)).
    \]
    By combining the previous considerations with \eqref{eq:poincaremonti} and a standard H\"older inequality we are able to conclude:
    \[
    \begin{aligned}
    &r^{2n+1}\fint_{\C_{r,U}(p)} |u-(u)_{\C_{r,U}(p)}|d\leb^{2n+1}=r^{2n+1}\int_\C |\tilde u-(\tilde u)_\C|d\leb^{2n+1}\\
    \leq\  & r^{2n+1}\left(\int_ \C|\tilde u-(\tilde u)_\C|^{\frac{2n+2}{2n+1}}d\leb^{2n+1}\right)^{\frac{2n+1}{2n+2}}\leq r^{2n+1}C_P|D_\H\tilde u|(\C)=C_P|D_\H u|(\C_{r,U}(p)).
    \end{aligned}
    \]
 \end{proof}

\section{Proof of Theorem \texorpdfstring{\ref{teo_intromain}}{1.1}}\label{sec_dimostrazione}
In this section we prove our main result, Theorem \ref{teo_intromain}. First, in Subsection \ref{subsec_red}, we reduce the problem to the class of \emph{good} $\sbv_\H$ functions defined in Definition \ref{def_goodsbv} and then we prove an upper bound (Theorem \ref{teo_upper}, in Subsection \ref{subsec_upper}) and a lower bound (Theorem \ref{teo_lower}, in Subsection \ref{subsec_lower}) for the functional $k_r(\cdot)$ on such class.

\subsection{Reduction to good SBV functions}\label{subsec_red}

In this subsection we show that it is enough to prove Theorem \ref{teo_intromain} for the class of \emph{good} $\sbv_\H$ functions. For the purpose, we recall the pivotal approximation result proved in \cite{sbvx}.

\begin{teo}[{\cite[Theorem 1.4]{sbvx}}]\label{teo_goodapprox}
    Let $\Omega \se \H^n$ be an open set and $u \in \sbv_\H(\Omega)$. Then there exists a sequence $(u_k)_{k \in \N}$ of good $\sbv_\H(\Omega)$ functions such that 
\[
\nl u-u_k \nr_{\bv_\H(\Omega)}\xrightarrow{k \to +\infty}0.
\]
\end{teo}

\begin{teo}\label{teo_bastanolegood}
   Let $\Omega \se \H^n$ be an open set. If Theorem \ref{teo_intromain} holds for good $\sbv_\H(\Omega)$ functions then it also holds  for generic $\sbv_\H(\Omega)$ functions.
\end{teo}
\begin{proof}
    Let $u \in \sbv_\H(\Omega)$ and consider a sequence $(u_k)_{k \in \N}$ of approximating good $\sbv_\H(\Omega)$ functions as per Theorem \ref{teo_goodapprox}. By assumption, for every $k \in \N$ 
\begin{equation}\label{eq_as1}
\lim_{r \to 0} k_r(u_k) = \frac{1}{4} \int_\Omega |D^\ap_\H u_k|d\leb^{2n+1}+\frac{1}{2}|D_\H^ju_k|(\J_{u_k}).
\end{equation}
Fix $r>0$. Given any Heisenberg $r$-cube $\Q \in \mathcal G_r$, we can estimate
\begin{align*}
\left| \fint_{\Q}| u-(u)_{\Q}|d\leb^{2n+1}- \fint_{\Q}| u_k-(u_k)_{\Q}|d\leb^{2n+1}\right|
& \leq 
\fint_{\Q} |(u-u_k)-(u-u_k)_{\Q}|d\leb^{2n+1}\\&\leq \frac{C_P}{r^{2n+1}}|D_\H(u-u_k)|(\Q),
\end{align*}
the last inequality coming from Theorem \ref{teo_poincare}. Thus, for every family $\mathcal G_r$ of disjoint Heisenberg $r$-cubes 
\[
r^{2n+1} \left| \sum_{\Q\in \mathcal G_r} \fint_{\Q} \left| u -(u)_{\Q}\right|d\leb^{2n+1}-\sum_{\Q\in \mathcal G_r} \fint_{\Q} \left| u_k -(u_k)_{\Q}\right|d\leb^{2n+1} \right| \leq C_P \sum_{\Q\in \mathcal G_r} |D_\H (u-u_k)|(\Q),
\]
hence 
\[
|k_r(u)-k_r(u_k)| \leq C_P  |D_\H (u-u_k)|(\Omega)\longrightarrow 0\qquad\text{as }k\to+\infty.
\]
The statement is now an immediate consequence of the previous inequality and~\eqref{eq_as1}.
\end{proof}

\subsection{Proof of the upper bound for good SBV functions}\label{subsec_upper}

\begin{lem}\label{lem:upper-smooth}
Let $\Omega \se \H^n$ be an open set and $u\in C^\infty(\Omega)\cap L^1(\Omega)$.  Then
\[
 \limsup_{r \to 0}k_r(u)
 \le\frac14\int_\Omega|\nabla_\H u|d\leb^{2n+1}.
\]
\end{lem}

\begin{proof}
We  assume $\int_\Omega |\nabla_\H u|d\leb^{2n+1}<+\infty$, otherwise  there is nothing to prove.
Fix $\eta>0$ and choose a compact set $K\subset\Omega$ such that
\begin{equation}\label{eq:tail-gradient}
\int_{\Omega\setminus K}|\nabla_\H u|d\leb^{2n+1}<\eta.
\end{equation}
Choose a bounded open set $V$ such that $K\subset \subset V\subset \subset\Omega$ and let $\mathfrak b>0$ be the dimensional constant in
\eqref{eq:cube-ball}, so that
\[
\diam(\C_{r,U}(p))\leq \mathfrak br.
\]
Let $r_0>0$ be such that every Heisenberg $r$-cube intersecting $K$ is contained in $V$. Applying Lemmata~\ref{lem_14xi} and~\ref{lem_taylor} to the compact set $\overline V$ we obtain for a suitable $C_V>0$
\begin{equation}\label{eq:upper-single-cube}
\fint_{\C_{r,U}(p)}|u-(u)_{\C_{r,U}(p)}|d\leb^{2n+1}\leq\frac r4|\nabla_\H u(p)|+C_Vr^2
\end{equation}
for every $p\in\overline V$, $U\in U(n)$ and $0<r<r_0$. Define
\[
\omega_V(r)\ceq\sup\left\{|\nabla_\H u(x)-\nabla_\H u(y)|:x,y\in\overline V,\ d(x,y)\leq \mathfrak b r\right\}.
\]
Since $\nabla_\H u$ is uniformly continuous on $\overline V$, we have
$
\omega_V(r)\xrightarrow{r\to0}0.
$
If $\Q=\C_{r,U}(p)\subset V$, then
\[
\begin{aligned}
\left||\nabla_\H u(p)|\leb^{2n+1}(\Q)-\int_{\Q}|\nabla_\H u(x)|d\leb^{2n+1}(x)\right|
&\leq\int_{\Q}\left||\nabla_\H u(p)|-|\nabla_\H u(x)|\right|d\leb^{2n+1}(x)
\\&\leq\omega_V(r)\leb^{2n+1}(\Q).
\end{aligned}
\]
Multiplying \eqref{eq:upper-single-cube} by $r^{2n+1}$ and using the above inequality gives
\begin{equation}\label{eq:upper-single-cube-integral}
r^{2n+1}\fint_{\Q}|u-(u)_{\Q}|d\leb^{2n+1}
\leq
\frac14\int_{\Q}|\nabla_\H u|d\leb^{2n+1}+\left(\frac14\omega_V(r)+C_Vr\right)\leb^{2n+1}(\Q).
\end{equation}
Let now $\mathcal X_r$ be a pairwise disjoint family of Heisenberg $r$-cubes contained in $\Omega$. We subdivide it as
\[
\mathcal X_r^K
\ceq
\{\Q\in\mathcal X_r:\Q\cap K\neq\emptyset\}
\qquad \text{ and } \qquad
\mathcal X_r^0\ceq\{\Q\in\mathcal X_r:\Q\cap K=\emptyset\}.
\]
For $r<r_0$, every Heisenberg cube in $\mathcal X_r^K$ is contained in $V$. Summing \eqref{eq:upper-single-cube-integral} and using pairwise disjointness, we obtain
\begin{equation}\label{eq:inside-K}
\begin{aligned}
&r^{2n+1}\sum_{\Q\in\mathcal X_r^K}\fint_{\Q}|u-(u)_{\Q}|d\leb^{2n+1}\\
&\qquad\leq
\frac14\int_{\bigcup_{\Q\in\mathcal X_r^K}\Q}|\nabla_\H u|d\leb^{2n+1}+\left(\frac14\omega_V(r)+C_Vr\right)\leb^{2n+1}(V)
\\
&\qquad\leq\frac14\int_\Omega|\nabla_\H u|d\leb^{2n+1}+\left(\frac14\omega_V(r)+C_Vr\right)\leb^{2n+1}(V).
\end{aligned}
\end{equation}
For every Heisenberg cube $\Q\in\mathcal X_r^0$, the Poincaré inequality of Theorem~\ref{teo_poincare} gives
\[
r^{2n+1}\fint_{\Q}|u-(u)_{\Q}|d\leb^{2n+1}\leq C_P|D_\H u|(\Q)=C_P\int_{\Q}|\nabla_\H u|d\leb^{2n+1}.
\]
Since the Heisenberg cubes in $\mathcal X_r^0$ are pairwise disjoint and do not intersect $K$, by \eqref{eq:tail-gradient} we obtain
\begin{equation}\label{eq:outside-K}
r^{2n+1}\sum_{\Q\in\mathcal X_r^0}\fint_{\Q}|u-(u)_{\Q}|d\leb^{2n+1}\leq C_P\int_{\Omega\setminus K}|\nabla_\H u|d\leb^{2n+1}\leq C_P\eta.
\end{equation}
Combining \eqref{eq:inside-K} and \eqref{eq:outside-K}, and then taking the supremum over all admissible families of Heisenberg $r$-cubes $\mathcal X_r$, we find
\[
k_r(u)\leq\frac14\int_\Omega|\nabla_\H u|d\leb^{2n+1}+\left(\frac14\omega_V(r)+C_Vr\right)\leb^{2n+1}(V)+C_P\eta.
\]
Letting $r\to0$ gives
\[
\limsup_{r\to0}k_r(u)\leq\frac14\int_\Omega|\nabla_\H u|d\leb^{2n+1}+C_P\eta,
\]
and the statement follows from the arbitrariness of $\eta>0$.
\end{proof}

\begin{teo}\label{teo_upper}
Let $\Omega \se \Hn$ be an open set and $u$ be a good $\sbv_\H(\Omega)$ function. Then
\[
\limsup_{r\to0}k_r(u)
\leq
\frac14\int_\Omega|D^\ap_\H u|\,d\leb^{2n+1}
+
\frac12|D_\H^ju|(\Omega).
\]
\end{teo}

\begin{proof}
Fix $\eta>0$ and let $K$ be as in Lemma \ref{lem_comb}. Fix $\delta>0$ and define 
    \[
    \Omega_\delta \ceq \{ x \in \Omega:d(x,\J_u)>\delta, d(x,\pp \Omega)>\delta\}.
    \]
    Clearly, $u$ is smooth on the open set $\Omega_\delta$. 
    Let $\mathcal X_r$ be a pairwise disjoint family of Heisenberg $r$-cubes contained in $\Omega$. Then define
    \[
   \mathcal X_r'\ceq \{ \Q \in \mathcal X_r: \Q \subset \Omega_\delta \}, \quad \mathcal X''_r\ceq \{ \Q \in \mathcal X_r \setminus \mathcal X'_r:\Q \cap K \neq \emptyset \}, \quad \mathcal X'''_r \ceq \mathcal X_r \setminus ( \mathcal X'_r \cup \mathcal X''_r )
    \]
By Lemma~\ref{lem:upper-smooth}, applied to the open set $\Omega_\delta$,
\begin{equation}\label{eq:upper-interior}
\limsup_{r\to0}\sup_{\mathcal X_r}r^{2n+1}\sum_{\Q\in\mathcal X_r'}\fint_{\Q}|u-(u)_{\Q}|\,d\leb^{2n+1}
\leq
\frac14\int_{\Omega_\delta}|\nabla_\H u|d\leb^{2n+1}.
\end{equation}

Fix $\ve>0$. By Lemma~\ref{lem_comb} $(vii)$, for all sufficiently small $r$
\begin{equation}\label{eq_estf2r}
r^{2n+1}\sum_{\Q\in\mathcal X_r''}\fint_{\Q}|u-(u)_{\Q}|d\leb^{2n+1}\leq\frac12|D_\H^ju|\left(\bigcup_{\Q\in\mathcal X_r''}\Q\right)+\ve r^{2n+1}\#\mathcal X_r''.
\end{equation}
We claim that there exists $C_K>0$, independent of $\mathcal X_r$, such that
\begin{equation}\label{eq_claimf2r}
\#\mathcal X_r''\leq C_Kr^{-(2n+1)}.
\end{equation}
In order to prove the latter, we first prove that there exist $c_K>0$ and $\rho_K>0$ such that
\begin{equation}\label{eq:uniform-jump-density}
|D_\H^ju|(q\cdot\delta_r\C)\geq c_Kr^{2n+1}
\end{equation}
for every $q\in K$ and every $0<r<\rho_K$. Otherwise, one could find $q_k\in K$ and $r_k\to 0$ such that
$
r_k^{-(2n+1)}|D_\H^ju|(q_k\cdot\delta_{r_k}\C)\to 0.
$
After passing to a subsequence, there exists $i$ such that $q_k\in K_i$ for every $k$ and $q_k\to q\in K_i$. Lemma~\ref{lem_comb} (vi), applied with $\C_k=\C$, would give
\[
\liminf_{k\to\infty}r_k^{-(2n+1)}|D_\H^ju|(q_k\cdot\delta_{r_k}\C)\geq|u^+(q)-u^-(q)|P_\H(H_{\nu_u(q)}^+;\C)>0,
\]
a contradiction. Hence, we proved \eqref{eq:uniform-jump-density}. 

Since every $\Q\in\mathcal X_r''$ intersects $K$ and $\diam(\Q)\leq \mathfrak b r$, one has
$
\cup_{\Q\in\mathcal X_r''}\Q\subset K_{\mathfrak b r}.
$
Now write $\Q=\C_{r,U_{\Q}}(p_{\Q})$ and choose $q_{\Q}\in\Q\cap K$. The balls
$B(p_{\Q},\mathfrak ar)$ are pairwise disjoint. Moreover, the inclusion
$
B(q_{\Q},\mathfrak br)\subset B(p_{\Q},2\mathfrak br)
$
implies that the family $\{B(q_{\Q},\mathfrak br):\Q\in\mathcal X_r''\}$ has uniformly bounded overlap, i.e., that there exists $M>0$, depending only on $n$, such that every point $q\in\H^n$ belongs at most to $M$ such balls. Using \eqref{eq:uniform-jump-density}, we obtain
\[
c_Kr^{2n+1}\#\mathcal X_r''\leq\sum_{\Q\in\mathcal X_r''}|D_\H^ju|(q_{\Q}\cdot\delta_r\C)\leq\sum_{\Q\in\mathcal X_r''}|D_\H^ju|(B(q_{\Q},\mathfrak br))\leq M|D_\H^ju|(\Omega),
\]
hence
\[
\#\mathcal X_r''\leq \frac{M|D_\H^ju|(\Omega)}{c_K} r^{-(2n+1)},
\]
proving \eqref{eq_claimf2r} with $C_K \ceq \frac{M|D_\H^ju|(\Omega)}{c_K} $. Consequently, from \eqref{eq_estf2r} we obtain
\begin{equation}\label{eq:upper-near-K}
r^{2n+1}
\sum_{\Q\in\mathcal X_r''}
\fint_{\Q}
|u-(u)_{\Q}|d\leb^{2n+1}
\leq
\frac12|D_\H^ju|(K_{\mathfrak br})
+\varepsilon C_K.
\end{equation}

Finally, let $\Q\in\mathcal X_r'''$. Since $\Q\not\subset\Omega_\delta$, there exists $x_{\Q}\in\Q$ such that
\[
\text{either}\quad
d(x_{\Q},\J_u)\leq\delta
\quad\text{or}\quad
d(x_{\Q},\partial\Omega)\leq\delta.
\]
Thus, if $r<\delta/\mathfrak b$, $\Q\subset\Omega\setminus\Omega_{2\delta}$. Moreover, $\Q\cap K=\emptyset$. By Theorem~\ref{teo_poincare} and pairwise disjointedness
\begin{equation}\label{eq:upper-remainder}
\begin{aligned}
r^{2n+1}\sum_{\Q\in\mathcal X_r'''}\fint_{\Q} |u-(u)_{\Q}|d\leb^{2n+1}
&\leq C_P|D_\H u|\left(\bigcup_{\Q\in\mathcal X_r'''}\Q \right)
\\&\leq C_P\left(\int_{\Omega\setminus\Omega_{2\delta}}|\nabla_\H u|\,d\leb^{2n+1}+|D_\H^ju|(\J_u\setminus K)\right).
\end{aligned}
\end{equation}
Taking the supremum over all admissible families of Heisenberg $r$-cubes, combining Lemma~\ref{lem_comb} (ii), \eqref{eq:upper-interior}, \eqref{eq:upper-near-K} and \eqref{eq:upper-remainder} and then letting $r\to0$ one finds
\[
\limsup_{r\to0}k_r(u) \leq \frac14 \int_{\Omega_\delta}|\nabla_\H u|\,d\leb^{2n+1}+\frac12|D_\H^ju|(K)+\varepsilon C_K+C_P\left( \int_{\Omega\setminus\Omega_{2\delta}} |\nabla_\H u|\,d\leb^{2n+1} +\eta \right).
\]
We first let $\ve \to 0$. Since $\J_u$ is compact, as $\delta\to 0$ one has $\Omega_\delta\to\Omega\setminus\J_u$ and $\Omega\setminus\Omega_{2\delta}\to \J_u$; using $\leb^{2n+1}(\J_u)=0$ we eventually find
\[
\limsup_{r\to0}k_r(u) \leq \frac14 \int_\Omega|D^\ap_\H u|\,d\leb^{2n+1}+\frac12|D_\H^ju|(K)+C_P\eta,
\]
and the statement follows from the arbitrariness of $\eta$.
\end{proof}

\subsection{Proof of the lower bound for good SBV functions}\label{subsec_lower}

\begin{lem}\label{lem_ar}
Let $\Omega \subseteq \mathbb H^n$ be an open set and let
$u$ be a good $\sbv_\H(\Omega)$ function. Fix a bounded open set
$V \subset\subset \Omega$ and, for every $\delta>0$, define
       \[
    V_\delta \ceq \{x \in V: d(x,\J_u)>\delta, d(x,\pp V)>\delta\}.
    \] 
For every $\delta,\tau>0$ and every $0<\rho<\frac{1}{2\sqrt{2n}}$, there exists $r_0>0$ such that, for every $r\in(0,r_0)$ there exists a family of Heisenberg $r$-cubes $\mathcal A_r \se \mathcal T_r$ with the following properties:
\begin{enumerate}
    \item[(i)] $\mathcal A_r$ is finite;
    \item[(ii)] the Heisenberg cubes in $\mathcal A_r$ are pairwise disjoint;
    \item[(iii)] every $\Q\in\mathcal A_r$ satisfies $\Q\subseteq V_\delta$.
\end{enumerate}
Moreover,
   \begin{equation}\label{eq_claimar}
   \begin{split}
    &\liminf_{r \to 0} r^{2n+1}\sum_{\Q \in \mathcal A_r}\fint_{\Q} |u-(u)_{\Q}|d\leb^{2n+1}\\
    \geq\ &\left( \frac{1}{4}-\frac{\sqrt{2n}}{2}\rho \right) \left(  \int_{V_\delta} |\nabla_\H u|d\leb^{2n+1}-\tau\leb^{2n+1}(V)                       \right).
    \end{split}
   \end{equation}
\end{lem}

\begin{proof}
    Since $u\in C^\infty(\Omega\setminus \J_u)$, we consider the open set $ W_\tau \ceq \{x \in V_\delta: |\nabla_\H u(x)|>\tau \}$ and the continuous function $g:W_\tau\to\mathbb S^{2n-1}$ defined by $g(x) \ceq {\nabla_\H u(x)}/{|\nabla_\H u(x)|}$.
For every $1 \leq h \leq 2n$, the finite measure 
    \[
    B \longrightarrow \leb^{2n+1}(\{x \in W_\tau: g_h(x) \in B\})
    \]
has at most countably many atoms. Therefore, we can choose finitely many real numbers 
\[
s_{h,0}<-1<s_{h,1}<\dots<s_{h,M_h-1}<1<s_{h,M_{h}}
\]
so that every $s_{h,\ell}$ is not an atom and $s_{h,\ell}-s_{h,\ell-1}< \frac{\rho}{\sqrt{2n}}$. Define $S_{h,\ell} \ceq \{ \nu \in \mathbb S ^{2n-1}: s_{h,\ell-1}<\nu_h< s_{h,\ell}  \}$ 
and  the sets $S_j$ as all possible nonempty intersections of the type $S_{1,\ell_1} \cap S_{2,\ell_2}  \cap \cdots \cap S_{2n,\ell_{2n}}$. Therefore we obtain finitely many nonempty, relatively open, pairwise disjoint sets $S_1,\dots,S_L$ such that every $S_j$ satisfies $\operatorname{diam}(S_j)<\rho$
and, because all the $s_{h,\ell}$ are not atoms, one has 
\begin{equation}\label{eq_est3}
\leb^{2n+1} \left( W_\tau\setminus \bigcup_{j=1}^L g^{-1}(S_j)     \right)=0.
\end{equation}

We define $A_j \ceq g^{-1}(S_j)$. The sets $\{A_j\}_{1 \leq j \leq L}$ are pairwise disjoint, bounded and open, and they cover $W_\tau$ up to a $\leb^{2n+1}$-negligible set. For every $1 \leq j\leq L$ we choose $\mu_j \in S_j$ and $U_j \in U(n)$ such that $U_je_1=\mu_j$. Recalling the notation $\mathcal P_{r,U}$ introduced in Remark~\ref{rem_mathcalP}, define
\[
\mathcal{A}_{r,j}\coloneqq\{ \Q \in\mathcal P_{r,U_j}: \Q \subset A_j \}\qquad\text{and} \qquad \mathcal A_r \coloneqq \bigcup_{j=1}^L \mathcal{A}_{r,j}.
\]
Since each $A_j$ is bounded and every tile in $\mathcal P_{r,U_j}$ has Lebesgue measure $r^{2n+2}$, the family $\mathcal A_{r,j}$ is finite. Moreover, the Heisenberg cubes belonging to $\mathcal A_{r,j}$ are pairwise disjoint, since $\mathcal P_{r,U_j}$ is a paving of $\mathbb H^n$. If $j\neq k$, every Heisenberg cube in $\mathcal A_{r,j}$ is contained in $A_j$, whereas every Heisenberg cube in $\mathcal A_{r,k}$ is contained in $A_k$. Since the sets $A_1,\ldots,A_L$ are pairwise disjoint, Heisenberg cubes belonging to different families cannot intersect. Consequently, $\mathcal A_r$ is a finite pairwise disjoint subfamily of $\mathcal T_r$. Finally, since $A_j\subseteq W_\tau\subseteq V_\delta$ for every $j=1,\ldots,L$, every Heisenberg cube in $\mathcal A_r$ is contained in $V_\delta$.

We are left to prove \eqref{eq_claimar}. Let $\Q=\C_{r,U_j}(p)\in\mathcal A_{r,j}$. Since $p\in\Q\subset A_j$,  $g(p)\in S_j$ and $|g(p)-\mu_j|<\rho$. Lemmata \ref{lem_auxacl} and \ref{lem_taylor} give
\begin{equation}\label{eq_est1}
\begin{split}
r^{2n+1}\fint_{\Q} |u-(u)_{\Q}|
&\geq r^{2n+2} \int_{\C} |\langle \nabla_\H u(p),U_jz \rangle|dzdt-2C_\delta r^{2n+3}\\
&\geq r^{2n+2}\left( \frac{1}{4}-\frac{\sqrt{2n}}{2}\rho \right) |\nabla_\H u(p)|-2C_\delta r^{2n+3}.
\end{split}
\end{equation}
Define $\omega_\delta(s) \ceq \sup \{ |\nabla_\H u(x)-\nabla_\H u(y)|: x,y \in \overline{V_\delta},\ d(x,y) \leq s    \}$; clearly, $\omega_\delta(s) \xrightarrow{s \to 0}0$ and $\operatorname{diam} (\Q) \leq \mathfrak b r$, where $\mathfrak b$ is as in \eqref{eq:cube-ball}. Since $\leb^{2n+1}(\Q)=r^{2n+2}$ we can write 
\[
r^{2n+2}|\nabla_\H u(p)| \geq \int_{\Q}|\nabla_\H u|d\leb^{2n+1}-\omega_\delta( \mathfrak br)r^{2n+2}
\]
which, combined with \eqref{eq_est1}, gives
\[
r^{2n+1}\fint_{\Q} |u-(u)_{\Q}|  \geq \left( \frac{1}{4}-\frac{\sqrt{2n}}{2}\rho \right) \int_{\Q}|\nabla_\H u|d\leb^{2n+1}-\left(   \left( \frac{1}{4}-\frac{\sqrt{2n}}{2}\rho \right)  \omega_\delta(\mathfrak br)+2C_\delta r                          \right)r^{2n+2}.
\]
By summing on all the Heisenberg cubes of $\mathcal A_r$ one finds 
\begin{align}\label{eq_est2}
r^{2n+1}\sum_{\Q \in \mathcal A_r}\fint_{\Q} |u-(u)_{\Q}|  \geq& \left( \frac{1}{4}-\frac{\sqrt{2n}}{2}\rho \right)  \int_{\bigcup_{\Q \in \mathcal{A}_{r}} \Q}|\nabla_\H u|d\leb^{2n+1}\\& \notag -\left(       \left( \frac{1}{4}-\frac{\sqrt{2n}}{2}\rho \right)  \omega_\delta( \mathfrak b r)+2C_\delta r                          \right)r^{2n+2}\#\mathcal A_r.
\end{align}
Since the Heisenberg cubes in  $\mathcal A_r$ are disjoint, $r^{2n+2}\# \mathcal A_r  \leq \leb^{2n+1}(W_\tau)$. By Lemma \ref{lem:pavage} one has 
\[
\lim_{r\to0}\leb^{2n+1}\left( A_j \setminus \bigcup_{\Q \in \mathcal A_{r,j}}\Q\right)=0.
\]
The horizontal gradient $|\nabla_\H u|$ is bounded on $\overline{V_\delta}$, hence taking the $\liminf$ as $r \to 0$ of \eqref{eq_est2} and recalling \eqref{eq_est3} 
\[
\liminf_{r \to 0} r^{2n+1}\sum_{\Q \in \mathcal A_r}\fint_{\Q} |u-(u)_{\Q}|  \geq \left( \frac{1}{4}-\frac{\sqrt{2n}}{2}\rho \right)  \int_{W_\tau}|\nabla_\H u|d\leb^{2n+1}.
\]
Eventually, by definition of $W_\tau$
\[
\int_{W_\tau}|\nabla_\H u|d\leb^{2n+1} \geq \int_{V_\delta}|\nabla_\H u|d\leb^{2n+1} -\tau \leb^{2n+1}(V),
\]
which is enough to conclude the proof.
\end{proof}

\begin{lem}\label{lem_fr}
    Let $\Omega \se \H^n$ be an open set and $u$ be a good $\sbv_\H(\Omega)$ function; for fixed $\eta>0$ and $0<\sigma<1$ let $K$, $r_0$ and $(\mathcal F_r)_{r\in(0,r_0)}$ be as in Lemma \ref{lem_comb}. Then
    \begin{equation}\label{eq_j1}
    \liminf_{r \to 0} r^{2n+1}\sum_{\Q \in \mathcal F_r} \fint_{\Q} |u-(u)_{\Q}|d\leb^{2n+1} \geq \frac{1}{2}(1-\sigma)(|D_\H^ju|(\J_u)-\eta)
    \end{equation}
    and
    \begin{equation}\label{eq_j2}
        \bigcup_{\Q \in \mathcal F_r}\Q \subset \{ x \in \H^n: d(x,K) \leq \mathfrak br\},
    \end{equation}
    where   $\mathfrak b>0$ is as in \eqref{eq:cube-ball}.
\end{lem}
\begin{proof}
Fix $r\in(0,r_0)$ and let $U_i,\ p^r_{i,h}$ and $\C_{i,h}^{r}=\C_{r,U_i}(p_{i,h}^r)$ be as in Lemma \ref{lem_comb}; call $\C_{i,h}^{r,\pm}\ceq \C^\pm_{r,U_i}(p^r_{i,h})$  the two halves of $\C_{i,h}^{r}$, as per Lemma \ref{lem_comb} $(iv)$. Consider
    \[
    m_{i,h}^{r,\pm}\ceq \fint_{\C_{i,h}^{r,\pm}}u d\leb^{2n+1},
    \qquad\text{so that}\qquad
    (u)_{\C_{i,h}^r}=\frac{ m_{i,h}^{r,+}+ m_{i,h}^{r,-}}{2}.
    \]
    Then
    \begin{equation}\label{eq_v1}
    \fint_{\C_{i,h}^r}\left|  u-(u)_{\C_{i,h}^r}     \right|d\leb^{2n+1} \geq \frac{1}{2}|m_{i,h}^{r,+}-(u)_{\C_{i,h}^r}   |  +  \frac{1}{2}|m_{i,h}^{r,-}-(u)_{\C_{i,h}^r}   |=\frac{1}{2}|m_{i,h}^{r,+}- m_{i,h}^{r,-}|
        \end{equation}
    and, if $\alpha(r)$ is as in Lemma \ref{lem_comb} $(iv)$, one has for each $i,h$
    \begin{equation}\label{eq_v2}
    | m_{i,h}^{r,+}-u^+(p^r_{i,h})|+ | m_{i,h}^{r,-}-u^-(p^r_{i,h})|\leq 2\alpha(r).
        \end{equation}
     Combining \eqref{eq_v1} and \eqref{eq_v2} we obtain
\[
\fint_{\C_{i,h}^r}  \left|  u-(u)_{\C_{i,h}^r}     \right|d\leb^{2n+1} \geq \frac{1}{2}  |u^+(p_{i,h}^r)-u^-(p_{i,h}^r)|-\alpha(r) .
\]
Multiplying by $r^{2n+1}$ the above inequality,  summing on all $h=1,\dots, N_i(r)$ and  using  Lemma \ref{lem_comb} $(viii)$ we obtain
\[
 r^{2n+1}\sum_{h=1}^{N_i(r)} \fint_{\C_{i,h}^r}  \left|  u-(u)_{\C_{i,h}^r}     \right|d\leb^{2n+1} \geq \frac{1}{2} (1-\sigma)|D_{\mathbb H}^j u|(K_i)-\alpha(r)r^{2n+1}N_i(r).
\]
Summing on  $i=1,\dots,m$ and  using again Lemma \ref{lem_comb} $(viii)$ we obtain
\begin{align*}
r^{2n+1}\sum_{\Q \in \mathcal F_r} \fint_{\Q}   \left|  u-(u)_{\Q}     \right|d\leb^{2n+1} &\geq \frac{1}{2} (1-\sigma)|D_{\mathbb H}^j u|(K)-\alpha(r)\sum_{i=1}^mr^{2n+1}{N_i(r)}\\
&\geq\frac{1}{2} (1-\sigma) (|D_\H^j u|(\J_u)-\eta)-\alpha(r)\sum_{i=1}^m r^{2n+1} N_i(r)
\end{align*}
By taking the $\liminf$ as $r \to 0$ and recalling that $\alpha(r) \xrightarrow{r \to 0}0$ and $\sup_{0<r<r_0}N_i(r)r^{2n+1}<+\infty$ (as per Lemma \ref{lem_comb} $(viii)(a)$) we obtain \eqref{eq_j1}. The inclusion \eqref{eq_j2} follows from \eqref{eq:cube-ball} and the fact that $p_{i,h}^r \in K$.
\end{proof}

\begin{teo}\label{teo_lower}
    Let $\Omega \se \Hn$ be an open set and $u$ be a good $\sbv_\H(\Omega)$ function. Then
    \[
    \liminf_{r \to 0}k_r(u) \geq \frac{1}{4}\int_\Omega |\nabla_\H u|d\leb^{2n+1}+\frac{1}{2}|D^j_\H u|(\J_u).
    \]
\end{teo}
\begin{proof}
    For $m \in \N$ define 
    \[
    U_m \ceq \{ x\in \Omega : d(0,x)<m,\ d(x, \H^n \setminus \Omega)>1/m\}.
    \]
    We have that $U_m$ is a bounded open subset of $\Omega$ such that $U_m \subset \subset \Omega$, the sequence $(U_m)_{m \in\N}$ is increasing and $U_m \xrightarrow{m \to +\infty}\Omega$. We can assume that $m$ is large enough so that $\J_u\cap U_m=\J_u$. 
    Fix   $\delta,\tau,\eta>0$, $0<\rho<\frac{1}{2\sqrt{2n}}$ and $0<\sigma<1$. 
    For every  sufficiently small $r>0$  let $\mathcal A_r$ be the family from Lemma \ref{lem_ar} (when applied with $V=U_m$) and $\mathcal F_r$ be the family from Lemma \ref{lem_fr} (applied with $\Omega=U_m$). We have that, whenever $\mathfrak br<\delta/2$ (where $\mathfrak b$ is as in the statement of Lemma \ref{lem_fr})
    \[
    \bigsqcup_{\Q \in \mathcal A_r}\Q \subset \{ x \in \H^n: d(x,\J_u)>\delta\}, \qquad     \bigsqcup_{\Q \in \mathcal F_r}\Q \subset \{ x \in \H^n: d(x,\J_u)<\delta/2\},
    \]
    hence the family of Heisenberg cubes $\mathcal G_r \ceq \mathcal A_r \cup \mathcal F_r$ is admissible for the functional $k_r$. Therefore 
    \[
    k_r(u) \geq r^{2n+1}\sum_{\Q \in \mathcal A_r}\fint_{\Q} |u-(u)_{\Q}|d\leb^{2n+1}+r^{2n+1}\sum_{\Q \in \mathcal F_r}\fint_{\Q} |u-(u)_{\Q}|d\leb^{2n+1}
    \]
and, by Lemmata \ref{lem_ar} and \ref{lem_fr} and the fact that   $\J_u\cap U_m=\J_u$,
    \begin{align*}
     \liminf_{r \to 0}k_r(u) &\geq     \left( \frac{1}{4}-\frac{\sqrt{2n}}{2}\rho \right) \left(  \int_{(U_m)_\delta}|\nabla_\H u|d\leb^{2n+1}-\tau \leb^{2n+1}(U_m)   \right)\\
     &\qquad+\frac{1}{2}(1-\sigma)(|D^j_\H u|(\J_u)-\eta).
    \end{align*}
    By letting  $\tau \to 0$, $\rho \to 0$, $\sigma \to 0$ and  $\eta \to 0$ we obtain from the above inequality  
        \begin{equation}\label{eq_est3bis}
     \liminf_{r \to 0}k_r(u) \geq     \frac{1}{4}   \int_{(U_m)_\delta}|\nabla_\H u|d\leb^{2n+1}   +\frac{1}{2}|D^j_\H u|(\J_u).
        \end{equation}
For fixed $m$ one has $(U_m)_\delta\uparrow U_m \setminus \J_u$ as $\delta \to 0$. Since $\J_u$ is closed and with zero Lebesgue measure, by monotone convergence we obtain
\[
\sup_{m , \delta} \int_{(U_m)_\delta}|\nabla_\H u|d\leb^{2n+1}=\int_\Omega |\nabla_\H u|d\leb^{2n+1},
\]
and the statement follows by taking the supremum in \eqref{eq_est3bis}.
\end{proof}

\bibliographystyle{acm}
\bibliography{BMObib}

\end{document}